\documentclass{article}

\usepackage[utf8]{inputenc}

\usepackage{bbm}
\usepackage{algorithm,algpseudocode}
\usepackage{epsfig}
\usepackage{authblk}
\usepackage{glossaries}
\usepackage{graphicx}
\usepackage{amsmath}
\usepackage{url}
\usepackage{amssymb}
\usepackage{amsfonts}
\usepackage{cite}
\usepackage{hyperref}
\usepackage{amsthm}

\newacronym{rbf}{RBF}{Radial Basis Function}
\newacronym{rbfs}{RBFs}{Radial Basis Functions}
\newacronym{rkhs}{RKHS}{Reproducing Kernel Hilbert Space}
\newacronym{pca}{PCA}{Principal Component Analysis}
\newacronym{pd}{PD}{positive definite}
\newacronym{spd}{SPD}{strictly positive definite}
\newacronym{pde}{PDE}{Partial Differential Equation}
\newacronym{pdes}{PDEs}{Partial Differential Equations}
\newacronym{svr}{SVR}{Support Vector Regression}

\newcommand{\R}{\mathbb{R}}
\newcommand{\N}{\mathbb{N}}
\newcommand{\Sp}[1]{\mbox{ span }\{#1\}}
\newcommand{\calh}{\mathcal H}
\newcommand{\fa}{\hbox{ for all }}
\newcommand{\inner}[2]{\ifthenelse{\equal{#2}{}}{\left\langle\cdot,\cdot\right\rangle_{#1}}{\left\langle#2\right\rangle_{#1}}}
\newcommand{\norm}[2]{\ifthenelse{\equal{#2}{}}{\left\|\cdot\right\|_{#1}}{\left\|#2\right\|_{#1}}}
\newcommand{\onevec}{\mathbbm{1}_n}
\DeclareMathOperator*{\argmin}{arg\min}
\DeclareMathOperator*{\argmax}{arg\max}

\newtheorem{theorem}{Theorem}
\newtheorem{proposition}[theorem]{Proposition}
\newtheorem{cor}[theorem]{Corollary}
\newtheorem{definition}[theorem]{Definition} 
\newtheorem{remark}[theorem]{Remark}
\newtheorem{example}[theorem]{Example}

\title{Kernel Methods for Surrogate Modeling}
\author[1]{G. Santin \thanks{gabriele.santin@mathematik.uni-stuttgart.de, \href{https://orcid.org/0000-0001-6959-1070}{orcid.org/0000-0001-6959-1070}}}
\author[1]{B. Haasdonk \thanks{bernard.haasdonk@mathematik.uni-stuttgart.de}}
\affil[1]{Institute for Applied Analysis and Numerical Simulation, University of Stuttgart, Germany}

\begin{document}

\maketitle 
  
\begin{abstract}
This chapter deals with kernel methods as a special class of techniques for surrogate modeling. Kernel methods have proven to be 
efficient in machine learning, pattern recognition and signal analysis due to their flexibility, excellent experimental performance and elegant functional analytic 
background. These data-based techniques provide so called kernel expansions, i.e., linear combinations of kernel functions which are generated from given input-output 
point samples that may be arbitrarily scattered. In particular, these techniques are meshless, do not require or depend on a grid, hence are less prone to the curse of 
dimensionality, even for high-dimensional problems. 

In contrast to projection-based model reduction, we do not necessarily  assume a high-dimensional model, but a general function that models input-output 
behavior 
within some simulation context. This could be some micro-model in a multiscale-simulation, some submodel in a coupled system, some initialization function for solvers, 
coefficient function in \gls*{pdes}, etc.

First, kernel surrogates can be useful if the input-output function is expensive to evaluate, e.g. is a result of a finite element simulation. Here, acceleration can 
be obtained by sparse kernel expansions. Second, if a function is available only via measurements or a few function evaluation samples, kernel approximation techniques 
can provide function surrogates that allow global evaluation.

We present some important kernel approximation techniques, which are kernel interpolation, greedy kernel approximation and support vector regression. Pseudo-code is 
provided for ease of reproducibility. In order to illustrate the main features, commonalities and differences, we compare these techniques on a real-world application. 
The experiments clearly indicate the enormous acceleration potential.
\end{abstract}

\section{Introduction}
This chapter deals with kernel methods as tools to construct surrogate models of arbitrary functions, given a finite set of arbitrary samples.
 
These methods generate approximants based solely on input-output pairs of the unknown function, without geometrical constraints on the sample locations.
In particular, the surrogates do not necessarily depend on the knowledge of an high-dimensional model but only on its observed input-output behavior at the sample sites, 
and they can be applied on arbitrarily 
scattered points in high dimension.

These features are particularly useful when these methods are applied within some simulation context. For example, kernel surrogates can be useful if the input-output 
function is expensive to evaluate, e.g. is a result of a finite element simulation. Here, acceleration can be obtained by sparse kernel expansions. 
Moreover, if a function is available only via measurements or a few function evaluation samples, kernel approximation techniques can provide function surrogates that 
allow global evaluation.

Kernel methods are used with much success in Model Order Reduction, and far beyond the scope of this chapter. For example, they have been used in the modeling 
of geometry 
transformations and mesh coupling \cite{BECKERT2001125,Deparis2014,Deparis2016}, and in mesh repair methods \cite{MARCHANDISE20122376}, or in the 
approximation of stability factors and error indicators \cite{MN14,ImmanuelThesis,Drohmann2015}, where only a few samples of the exact indicators are sufficient to 
construct an efficient surrogate to be used in the online phase. Moreover, kernel methods have been combined with projection based MOR methods, e.g. to 
obtain simulation-based classification \cite{TPYP2016}, or to derive multi-fidelity Monte Carlo approximations \cite{Peherstorfer2018a}. Kernel surrogates 
have been employed in optimal control problems \cite{Schmidt2018f,Suykens2001}, in the coupling of multi-scale simulations in biomechanics \cite{Wirtz2015a, SH2017b}, 
in real time prediction for parameter identification and state estimation in biomechanical systems \cite{KSHH2017},
in gas transport problems \cite{Grundel2013}, in the reconstruction of potential energy surfaces \cite{Kowalewski2016}, in the forecasting of time stepping methods 
\cite{Bruennette2019}, in the reduction of nonlinear dynamical systems \cite{Wirtz2012a}, in uncertainty quantification \cite{Koeppel2018}, and 
for nonlinear balanced truncation of dynamical systems \cite{Bouvrie2017}. 

In further generality, there exists many kernel-based algorithms and application fields that we do not address here. Mainly, we mention the solution of \gls*{pdes}, in 
which several approaches have emerged in the last years, and which particularly allow to solve problems with unstructured grids on general 
geometries, including high dimensional manifolds (see e.g. \cite{FornbergFlyer2015, Chen2014}). Moreover, several other techniques are studied within Machine Learning, 
such as classification, density estimation, novelty detection or feature extraction (see e.g. \cite{SS02,Shawe-Taylor2004}).

Furthermore, we remark that these methods are members of the larger class of machine learning and approximation techniques, which are generally suitable to construct 
models based on samples to make prediction on new inputs. These models are usually referred to as surrogates when they are then used as replacements of the model that 
generated the data, as they are able to provide an accurate and faster response. Some examples of these techniques are classical approximation methods such as polynomial 
interpolation, which are used in this context especially in combination with sparse grids to deal with high-dimensional problems (see \cite{Garcke2012}), and (deep) 
neural network models. The latter in particular have seen a huge increase in analysis and application in the recent years. For a recent treatment of deep learning, we 
refer e.g. to \cite{Goodfellow2016}.

Despite these very diverse applications and methodologies, kernel methods can be analyzed to some extent in the common framework of Reproducing Kernel Hilbert 
spaces and, 
although the focus of this chapter will be on the construction of sparse surrogate models, parts of the following discussion can be the starting point for the 
analysis of other techniques. 

In general terms, kernel methods can be viewed as nonlinear versions of linear algorithms. 
As an example, assume to have some set $X_n:=\{x_k\}_{k=1}^n\subset\R^d$ of data points and target data values 
$Y_n:=\{y_k\}_{k=1}^n\subset\R$.
We can construct a surrogate $s:\R^d\to\R$ that predicts new data via linear regression, i.e., find $w\in\R^d$ s.t. $s(x) := \inner{}{w, x}$, where $\inner{}{}$ is 
the scalar product in $\R^d$. A good surrogate model $s$ will give predictions such that $|s(x_k) - y_k|$ is small. If we can write $w\in\R^d$ 
as 
$
w= \sum_{j=1}^n \alpha_j x_j
$
for a set of coefficients $\left(\alpha_i\right)_{i=1}^n\in\R^n$, then $s$ can be rewritten as 
\begin{align*}
s(x) := \sum_{j=1}^n \alpha_j \inner{}{x_j, x}.
\end{align*}
Note that this formulation includes also regression with an offset (or bias) $b\neq 0$, which can be written in this form by an extended representation as
\begin{align*}
s(x) := \inner{}{w, x} + b =:\inner{}{\bar w, \bar x},
\end{align*}
where $\bar x:=(x, 1)^T\in\R^{d+1}$ and $\bar w:=(w, b)^T\in\R^{d+1}$.

Using now the Gramian matrix $A\in\R^{n\times n}$ with entries $A_{ij}:=\inner{}{x_i, x_j}$ and rows $A_i^T\in\R^n$, we look for the surrogate $s$ which minimizes 
\begin{align*}
\sum_{i=1} ^n\left(s(x_i) - y_i \right)_2^2 = \sum_{i=1}^n \left(A_i^T \alpha - y_i \right)_2^2 = \left\|A\alpha - y\right\|_2^2.
\end{align*}
Additionally, a regularization term can be added to keep the norm of $\alpha$ small, e.g. in terms of the value $\alpha^T A \alpha$.
Thus, the surrogate can be characterized as the solution of the optimization problem 
\begin{align*}
\min_{\alpha\in\R^n} \left\|A\alpha - y\right\|_2^2 + \lambda  \alpha^T A \alpha,
\end{align*}
i.e., $\alpha = (A + \lambda I)^{-1} y$ if $\lambda>0$.

In many cases this (regularized) linear regression is not sufficient to obtain a good surrogate. A possible idea is to try to combine this linear, simple method with a 
nonlinear function which maps the data to a higher dimensional space, where the hope is that the image of the data can be processed linearly. For this we consider a 
so-called feature map $\Phi:\R^d\to H$, where $H$ is a Hilbert space, and apply the same algorithm to the transformed data $\Phi(X_n):=\{\Phi(x_i)\}_{i=1}^n$ 
with the same values $Y_n$. Since the algorithm depends on $X_n$ only via the Gramian $A$, it is sufficient to replace it with the new Gramian 
$A_{ij}:=\inner{H}{\Phi(x_i), \Phi(x_j)}$ to obtain a nonlinear algorithm. 

We will see that $\inner{H}{\Phi(x), \Phi(y)}$ defines in fact a positive definite 
kernel, and if any numerical procedure can be written in terms of inner products of the inputs, it can be transformed in the same way into a new nonlinear algorithm 
simply by replacing the inner products with kernel evaluations (the so-called kernel trick). We will discuss the details of this procedure in 
the next sections in the case of interpolation and Support Vector Regression, but this immediately gives a glance of 
the ample spectrum of algorithms in the class of kernel methods.

This chapter is organized as follows. Section \ref{sec:background} covers the basic notions on kernels and kernel-based spaces which are necessary for the 
development and understanding of the algorithms. The next Section \ref{sec:surrogates_general} presents the general ideas and tools to construct kernel 
surrogates as characterized by the Representer Theorem, and these ideas are specialized to the case of kernel interpolation in Section \ref{sec:interp} and Support 
Vector Regression in Section \ref{sec:svr}. In both cases, we provide the theoretical foundations as well as the algorithmic description of the methods, with 
particular attention to techniques to enforce sparsity in the model. These surrogates can be used to perform various analyses of the full model, and we give some 
examples in Section \ref{sec:model_analysis}. Section \ref{sec:parameter_selection} presents a general strategy to choose the various 
parameters defining the model, whose tuning can be critical for a successful application of the algorithms. Finally, we discuss in Section \ref{sec:numerics} 
the numerical results of the methods on a real application dataset, comparing training time (offline), prediction time (online), and accuracy.

\section{Background on kernels}\label{sec:background}
We start by introducing some general facts of positive definite kernels. Further details on the general analytical theory of reproducing kernels can be found 
e.g. in the 
recent monograph \cite{Saitoh2016}, while the books \cite{Fasshauer2015,Wendland2005} and \cite{SS02,Steinwart2008} contain a treatment of kernel theory from 
the point of view of pattern analysis and scattered data interpolation, respectively. 

\subsection{Positive definite kernels}
Given a nonempty set $\Omega$, which can be a subset of $\R^d$, $d\in\N$, but also a set of structured objects such as strings or graphs, a real- and scalar-valued 
kernel $K$ on $\Omega$ is a bivariate symmetric function $K:\Omega\times\Omega\to\R$, i.e., $K(x, y) = K(y, x)$ for all $x, y\in\Omega$. For our purposes, we are 
interested in (strictly) positive definite kernels, defined as follows.

\begin{definition}[Positive definite kernels]\label{def:posdefker}
Let $\Omega$ be a nonempty set. A kernel $K$ on $\Omega$ is \gls*{pd} on $\Omega$ if for all $n\in\N$ and for any set of $n$ pairwise distinct elements 
$X_n:=\{x_i\}_{i=1}^n\subset \Omega$, the kernel matrix (or Gramian matrix) $A:=A_{K, X_n}\in\R^{n\times n}$ defined as $A_{ij}:=K(x_i, x_j)$, $1\leq i,j\leq 
n$, is 
positive semidefinite, i.e., for all vectors $\alpha := \left(\alpha_i\right)_{i=1}^n\in \R^n$ it holds
\begin{equation}\label{eq:def_of_pd}
\alpha^T A \alpha = \sum_{i,j=1}^n \alpha_i\alpha_j K(x_i, x_j) \geq 0.
\end{equation}
The kernel is \gls*{spd} if the kernel matrix is positive definite, i.e., \eqref{eq:def_of_pd} holds with strict inequality when $\alpha\neq 0$.
\end{definition}

The further class of conditionally (strictly) positive definite kernels is also of interest in certain contexts. We refer to \cite[Chapter 8]{Wendland2005}
 for their extensive treatment, and we just mention that they are defined as above, except that the condition \eqref{eq:def_of_pd} has to be satisfied only for 
the subset of coefficients $\alpha$ which match a certain orthogonality condition. When this condition is defined with respect to a space of polynomials of degree 
$m\in\N$, 
the resulting kernels are used e.g. to guarantee a certain polynomial exactness of the given approximation scheme, and they are often employed in 
certain methods for the solution of \gls*{pdes}. 

\subsection{Examples and construction of kernels}\label{sec:construct_kernels}
Despite the abstract definition, there are several ways to construct functions $K:\Omega\times\Omega\to\R$ which are (strictly) positive definite kernels, and usually 
the proper choice of the kernel is a crucial step in the successful application of the method. We list here a general strategy to construct kernels, and some notable 
examples.

An often used, constructive approach to design a new kernel is via feature maps as follows.
\begin{proposition}[Kernels via feature maps]
Let $\Omega$ be a nonempty set. A feature map $\Phi$ is any function $\Phi:\Omega\to H$, where $\left(H, \inner{H}{}\right)$ is any Hilbert space (the feature space). 
The function 
\begin{align*}
K(x,y):= \inner{H}{\Phi(x), \Phi(y)}\;\;x,y\in \Omega,
\end{align*}
is a \gls*{pd} kernel on $\Omega$.
\end{proposition}
\begin{proof}
$K$ is a \gls*{pd} kernel since it is symmetric and positive definite, because the inner product is bilinear, symmetric and positive definite.
\end{proof}
In many cases, $H$ is either $\R^m$ with very large $m$ or even an infinite dimensional Hilbert space. The computation of the possibly expensive $m$- or 
infinite-dimensional inner product can be 
avoided if a closed form for $K$ can be obtained. This implies a significant reduction of the computational time required to evaluate the kernel and thus to execute 
any kind of algorithm.

We see now some examples.
\begin{example}[Expansion kernels]
The construction comprises finite dimensional linear combinations, i.e., for a set of functions $\{v_j\}_{j=1}^m:\Omega\to\R$, the function 
$K(x, y) := \sum_{k=1}^m v_k(x) v_k(y)$ is a positive definite kernel, having a feature map
\begin{align}\label{eq:finite_feature_map}
\Phi(x):=\left(v_1(x), v_2(x), \dots, v_m(x)\right)^T\in H:=\R^m.
\end{align}
This idea can be extended to an infinite number of functions provided $\{v_j(x)\}_{j=1}^\infty\in H:=\ell_2(\N)$ uniformly in $\Omega$, and the resulting 
kernels are 
called 
Hilbert-Schmidt or expansion kernels, which can be proven to be even \gls*{spd} under additional conditions (see \cite{Schaback2002a}).
As an example in $d=1$, we mention the Brownian Bridge kernel $K(x, y) := \max(x, y) - xy$, defined with a feature map $v_j(x):= \sqrt{2} (j\pi)^{-1}\sin(j \pi x)$ for 
$j\in\N$, which is SPD on $\Omega:=(0,1)$. We remark that the kernel can be extended to $(0,1)^d$ with $d>1$ using a tensor product of one-dimensional kernels.

This feature map representation proves also that $dim(H)=:m<\infty$ means that the kernel is not \gls*{spd} in general: e.g., if $X_n$ contains $n$ pairwise 
distinct points and $m<n$, then the vectors $\{\Phi(x_i)\}_{i=1}^n$ can not be linearly independent, and thus the kernel matrix is singular.
\end{example}

\begin{example}[Kernels for structured data]
Feature maps are also employed to construct positive definite kernels on sets $\Omega$ of structured data, such as sets of strings, graphs, or any other object. 
For example, the convolution kernels introduced in \cite{Gaertner2003, Haussler1999} consider a finite set of features $v_1(x), \dots, 
v_m(x)\in\R$ of an object $x\in\Omega$, and define a feature map exactly as in \eqref{eq:finite_feature_map}. 
\end{example}

\begin{example}[Polynomial kernels]
For $a\geq 0$, $p\in\N$, $x,y\in\R^d$, the polynomial kernel
\begin{align}\label{eq:poly_kernel}
K(x,y):=\left(\inner{}{x,y} + a\right)^p = \left(\sum_{i=1}^d x^{(i)} y^{(i)} + a\right)^p,\;\;x:=\left(x^{(1)},\dots, x^{(d)}\right)^T, 
\end{align}
is \gls*{pd} on any $\Omega\subset \R^d$. It is a $d$-variate polynomial of degree $p$, which contains the monomial terms of degrees  $j:=\left(j^{(1)}, \dots, 
j^{(d)}\right)\in J$, for a certain set $J\subset \N_0^d$. If $m:=|J|$, a feature space is $\R^{m}$ with feature map
$$
\Phi(x):=\left(\sqrt{a_{1}} x^{j_1}, \dots, \sqrt{a_{m}} x^{j_m}\right)^T.
$$
for some positive numbers $\{a_j\}_{j=1}^m$ and monomials $x^{j_m}:=\prod_{i=1}^d (x^{(i)})^{j_m^{(i)}}$.

Observe that using the closed form \eqref{eq:poly_kernel} of the kernel instead of the feature map is very convenient, since we work with $d$-dimensional instead of 
$m$-dimensional vectors, 
where possibly $m:=|J| =  {{d + p}\choose{d}}=\dim(\mathbb P_p(\R^d))\gg d$.
\end{example}

\begin{example}[RBF kernels]
For $\Omega\subset\R^d$ in many applications the most used kernels are translational invariant kernels, i.e., there exist a function $\phi: \R^d 
\to\R$ with
\begin{align*}
K(x, y) := \phi(x - y), x, y \in\Omega,
\end{align*}
and in particular radial kernels, i.e., there exist a univariate function $\phi: \R_{\geq 0} \to\R$ with
\begin{align*}
K(x, y) := \phi(\|x - y\|), x, y \in\Omega.
\end{align*}
A radial kernel, or \gls*{rbf}, is usually defined up to a shape parameter $\gamma > 0$  that controls the
scale of the kernel via $K(x, y) := \phi(\gamma\|x - y\|)$. 

The main example of such kernels is the Gaussian $K(x, y) := e^{-\gamma^2 \|x-y\|^2}$, which is in fact strictly positive definite.
An explicit feature map has been computed in \cite{Steinwart2006}: If $\Omega\subset\R^d$ is nonempty, a feature map is the function 
$\Phi_{\gamma}:\Omega\to L_2(\R^d)$ defined by  
\begin{align*}
\Phi_{\gamma}(x):= \frac{(2\gamma)^{\frac{d}{2}}}{\pi^{\frac{d}{4}}} \exp\left(-2\gamma^2 \|x-\cdot\|^2\right),\;\;x\in\Omega.
\end{align*}
In this case it is even more evident how working with the closed form of $K$ is much more efficient than working with a feature map and computing $L_2$-inner products.

\gls*{rbf} kernels offer a significant easiness of implementation in arbitrary space dimension $d$. The evaluation of the kernel $K(\cdot, x)$, $x\in\R^d$, on a vector 
of $n$ points can indeed by realized by first computing a distance vector $D\in\R^n$, $D_i := \|x-x_i\|$, and then applying the univariate function $\phi$ on 
$D$. A 
discussion and comparison of different algorithms (in Matlab) to efficiently compute a distance matrix can be found in \cite[Chapter 4]{Fasshauer2015}, and most 
scientific computing languages comprise a built-in implementation (such as \verb+pdist2+\footnote{\url{https://www.mathworks.com/help/stats/pdist2.html}} 
in Matlab and {\verb+distance_matrix+}\footnote{\url{https://docs.scipy.org/doc/scipy/reference/generated/scipy.spatial.distance_matrix.html}} in Scipy).

Translational invariant and \gls*{rbf} kernels can be often analyzed in terms of their Fourier transforms, which provide proofs of their strict positive 
definiteness via the Bochner Theorem (see e.g. \cite[Chapter 6]{Wendland2005}), and connections to certain Sobolev spaces, as we will briefly see in Section 
\ref{sec:rkhs}.

Among various \gls*{rbf} kernels, there are also compactly supported kernels, i.e., $K(x,y)= 0$ if $\|x-y\|>1/\gamma$, which produce sparse kernel 
matrices if $\gamma$ is large enough. The most used ones are the Wendland kernels introduced in \cite{Wendland1995a}, which are even radial polynomial 
within their support.
\end{example}

There are, in addition,  various operations to combine positive definite kernels and obtain new ones. For example, sums and 
products of positive definite kernels and multiplication by a positive constant $a>0$ produce again positive definite kernels. 
Moreover, if $K'$ is a positive definite kernel and $K''$ is symmetric with $K'\preccurlyeq K''$ (i.e., $K:= K'' - K'$  is PD) then also $K''$ is positive definite. 
Furthermore, if $\Omega=\Omega'\times\Omega''$  and $K'$, $K''$ are \gls*{pd} kernels on $\Omega'$, $\Omega''$, then 
$K(x, y):= K'(x',y')K''(x'',y'')$ and $K(x, y):= K'(x',y')+K''(x'',y'')$ are also \gls*{pd} kernels on $\Omega$, i.e., kernels can be defined to respect tensor product 
structures of the input. 

Further details and examples can be found in \cite[Chapters 1--2]{Saitoh2016}.

\subsection{Kernels and Hilbert spaces}\label{sec:rkhs}
Most of the analysis of kernel-based methods is possible through the connection with certain Hilbert spaces. We first give the following definition.

\begin{definition}[Reproducing Kernel Hilbert Space]\label{def:rkhs}
Let $\Omega$ be a nonempty set, $\calh$ an Hilbert space of functions $f:\Omega\to\R$ with inner product $\inner{\calh}{}$. Then $\calh$ is called a 
\gls*{rkhs} on $\Omega$ if there exists a function $K:\Omega\times\Omega\to\R$ (the reproducing kernel) such that 
\begin{enumerate}
 \item\label{prop:k_in_rkhs} $K(\cdot, x)\in \calh$ for all $x\in \Omega$,
 \item\label{prop:reproducing} $\inner{\calh}{f, K(\cdot, x)} = f(x)$ for all $x\in \Omega$, $f\in \calh$ (reproducing property).
\end{enumerate}
\end{definition}

The reproducing property is equivalent to state that, for $x\in\Omega$, the $x$-translate  $K(\cdot, x)$ of the kernel is the Riesz representer of the evaluation
functional $\delta_x:\calh\to\R$, $\delta_x(f):=f(x)$ for $f\in\calh$ , that is hence a continuous functional in $\calh$. Also the converse holds, and the following 
result gives an abstract criterion to 
check if a Hilbert space is a \gls*{rkhs}.
\begin{theorem}
An Hilbert space of functions $\Omega\to\R$ is a \gls*{rkhs} if and only if the point evaluation functionals are continuous in $\calh$ for all $x\in\Omega$, i.e., 
$\delta_x\in\calh'$, the dual space of $\calh$. Moreover, the reproducing kernel $K$ of $\calh$ is strictly positive definite if and only if the functionals $\{\delta_x: 
x\in\Omega\}$ are linearly 
independent in $\calh'$.
\end{theorem}
\begin{proof}
The first part is clear from the reproducing property, while strict positive definiteness can be checked by verifying that the quadratic form in Definition 
\ref{def:posdefker} can not be zero for $\alpha\neq 0$ if $\{\delta_x: x\in\Omega\}$ are linearly independent.
\end{proof}

We see two concrete examples.
\begin{example}[Finite dimensional spaces]
Any finite dimensional Hilbert space $\calh$ of functions on a nonempty set $\Omega$ is a \gls*{rkhs}. If $m:=\dim(\calh)$ and 
$\{v_j\}_{j=1}^m$ is an orthonormal basis, then a reproducing kernel is given by
\begin{align*}
K(x,y) := \sum_{j=1}^mv_j(x) v_j(y), \; x,y\in\Omega.
\end{align*}
Indeed, the two properties of Definition \ref{def:rkhs} can be easily verified by direct computation.
\end{example}

\begin{example}[The Sobolev space $H_0^1(0,1)$]
The Sobolev space $H_0^1(0,1)$ with inner product $\inner{H_0^1}{f, g} := \int_0^1 f'(y) g'(y) dy$ is a \gls*{rkhs} with the Brownian Bridge kernel 
\begin{align*}
K(x,y):=\min(x, y) - x y, \;\;x,y\in(0,1)
\end{align*}
as reproducing kernel (see e.g. \cite{CaFaMcC2014}). Indeed, $K(\cdot, x)\in H_0^1(0,1)$, and the reproducing property \eqref{prop:reproducing} follows by 
explicitly computing the inner product.
\end{example}

The following result proves that reproducing kernels are in fact positive definite kernels in the sense of Definition \ref{def:posdefker}.
Moreover, the first two properties are useful to deal with the various type of approximants of Section \ref{sec:interp} and Section \ref{sec:svr}, which 
will be 
exactly of this form.

\begin{proposition}[]\label{prop:two_rkhs}
Let $\calh$ be a RKHS on $\Omega$ with reproducing kernel $K$. 
Let $n, n'\in\N$, $\alpha\in\R^n$, $\alpha'\in\R^{n'}$,  $X_n, X_{n'}' \subset \Omega$, and define the functions
$$
f(x):=\sum_{i=1}^n \alpha_i K(x, x_i),\;\; g(x):=\sum_{j=1}^{n'} \alpha_j' K(x, x_j'),\;\; x\in\Omega.
$$
Then we have the following:
\begin{enumerate}
\item\label{prop:f_in_h} $f, g \in \calh$,
\item\label{prop:f_inner_g} $\inner{\calh}{f, g} = \sum_{i=1}^n\sum_{j=1}^{n'} \alpha_i\alpha_j' K(x_i, x_j')$.
\item\label{prop:rk_is_pd_and_unique} $K$ is the unique reproducing kernel of $\calh$ and it is a positive definite kernel. 
\end{enumerate}
\end{proposition}
\begin{proof}
The first two properties follow from Definition \ref{def:rkhs}, and in particular from $\calh$ being a linear space and from the bilinearity of $\inner{\calh}{}$.

For Property \eqref{prop:rk_is_pd_and_unique}, the fact that $K$ is symmetric and positive definite, hence a \gls*{pd} kernel, follows from Property 
\eqref{prop:k_in_rkhs} of Definition 
\ref{def:rkhs}, and from the symmetry and positive definiteness of the inner product. Moreover, the reproducing property implies that, if $K, K'$ are two reproducing 
kernels of $\calh$, then for all $x,y\in\Omega$ it holds 
\begin{align*}
K(x, y) = \inner{\calh}{K(\cdot, y), K'(\cdot, x)} = K'(x, y).
\end{align*}
\end{proof}

It is common in applications to follow instead the opposite path, i.e., to start with a given \gls*{pd} kernel, and try to see if an appropriate \gls*{rkhs} 
exists. 
This is in fact always the case, as proven by the following fundamental theorem from \cite{Aronszajn1950}.

\begin{theorem}[RKHS from kernels -- Moore-Aronszajn Theorem]\label{th:aronszajn}
Let $\Omega$ be a nonempty set and $K:\Omega\times\Omega\to \R$ a positive definite kernel. Then there exists a unique \gls*{rkhs} $\calh:=\calh_K(\Omega)$ with 
reproducing kernel $K$. 
\end{theorem}
\begin{proof}
The theorem was first proven in \cite{Aronszajn1950}, to which we refer for a detailed proof. 
The idea is to deduce that, by Property \eqref{prop:f_in_h} of Proposition \ref{prop:rk_is_pd_and_unique}, a candidate \gls*{rkhs} $\calh$ of $K$ needs to contain the 
linear space
\begin{align*}
 \calh_0:=\Sp{K(\cdot, x): x\in\Omega}
\end{align*}
of finite linear combinations of kernel translates. Moreover, from Property \eqref{prop:f_inner_g} of Proposition \ref{prop:two_rkhs}, the inner product on 
this $\calh_0$ needs to satisfy 
\begin{align}\label{eq:bilinear_form}
\inner{\calh}{f, g} &= \sum_{i=1}^n\sum_{j=1}^{n'} \alpha_i\alpha_j' K(x_i, x_j').
\end{align}
With this observation in mind, the idea of the construction of $\calh$ is to start by $\calh_0$, prove that \eqref{eq:bilinear_form} defines indeed an inner product on 
$\calh_0$, and that the completion of $\calh_0$ w.r.t. this inner product is a \gls*{rkhs} having $K$ as reproducing kernel. Uniqueness then follows from 
Property \eqref{prop:rk_is_pd_and_unique} of the same proposition. 
\end{proof}
As it is common in the approximation literature, we will sometimes refer to this unique $\calh$ as the native space of the kernel $K$ on $\Omega$.

\begin{remark}[Kernel feature map]\label{prop:kernel_feature_map}
Among other consequences, this construction allows to prove that any \gls*{pd} kernel is generated by at least one feature map. Indeed, the function 
$\Phi: \Omega\to \calh$, $\Phi(x) := K(\cdot, x)$, is clearly a feature map for $K$ with feature space $\calh$, since the reproducing property implies that
\begin{align*}
\inner{H}{\Phi(x),\Phi(y)} = \inner{\calh}{K(\cdot, x),K(\cdot, y)} = K(x, y)\;\;\fa x, y\in\Omega.
\end{align*}

\end{remark}

\begin{remark}
For certain translational invariant kernels it is possible to prove that the associated native space is norm equivalent to a Sobolev spaces of the appropriate 
smoothness, which is related to the kernels' smoothness (see \cite[Chapter 10]{Wendland2005}). This is particularly interesting since the approximation 
properties of the different algorithms, including certain optimality that we will see in the next sections, are in fact optimal in these Sobolev spaces (with an 
equivalent norm). 
\end{remark}

The various operations on positive definite kernels mentioned in Section \ref{sec:construct_kernels} have an analogous effect on the corresponding native 
spaces. 
For example, the scaling by a positive number $a>0$ does not change the native space, but scales the inner product correspondingly, and, if 
$K' \preccurlyeq K''$ are positive definite kernels, then ${\mathcal H}_{K'}(\Omega)  \subset  {\mathcal H}_{K''}(\Omega)$. We remark that the latter property has been 
used for 
example in \cite{ZhZh2013} to prove inclusion relations for the native spaces of RBF kernels with different shape parameters.

\subsection{Kernels for vector-valued functions}\label{sec:mat_val}
So far we only dealt with scalar-valued kernels, which are suitable to treat scalar-valued functions. Nevertheless, it is 
clear 
that the interest in model reduction is typically also on vector-valued or multi-output functions, which thus require a generalization of the theory presented so far.
This has been done in \cite{Micchelli2005}, and it is based on the following definition of matrix-valued kernels.
\begin{definition}[Matrix-valued \gls*{pd} kernels]
Let $\Omega$ be a nonempty set and $q\in\N$. A function $K:\Omega\times\Omega\to \R^{q\times q}$ is a matrix valued kernel if it is symmetric, i.e., $K(x, y) = K(y, 
x)^T$ 
for all $x, y\in\Omega$. It is a \gls*{pd} (resp., \gls*{spd}) matrix-valued kernel if the kernel matrix $A\in\R^{nq\times nq}$ is positive semidefinite (resp., positive 
definite) for all $n\in\N$ and for all sets $X_n\subset\Omega$ of pairwise distinct elements.
\end{definition}
This more general class of kernels is also associated to a uniquely defined native space of vector-valued functions, where the notion of \gls*{rkhs} is replaced by 
the following.
\begin{definition}[\gls*{rkhs} for matrix-valued kernels]
Let $\Omega$ be a nonempty set, $q\in\N$, $\calh$ an Hilbert space of functions $f:\Omega\to\R^q$ with inner product $\inner{\calh}{}$. Then $\calh$ is called a 
vector-valued \gls*{rkhs} on $\Omega$ if there exists a function $K:\Omega\times\Omega\to\R^{q\times q}$ (the matrix-valued reproducing kernel) such that 
\begin{enumerate}
 \item $K(\cdot, x)v\in \calh$ for all $x\in \Omega$, $v\in\R^q$,
 \item $\inner{\calh}{f, K(\cdot, x)v} = f(x)^Tv$ for all $x\in \Omega$, $v\in\R^q$, $f\in \calh$ (directional reproducing property).
\end{enumerate}
\end{definition}

A particularly simple version of this construction can be realized by considering separable matrix-valued kernels (see e.g. \cite{Alvarez2012}), i.e., kernels that are 
defined as 
$
K(x, y) := \tilde K(x, y) B
$, where $\tilde K$ is a standard scalar-valued \gls*{pd} kernel, and $B\in\R^{q\times q}$ is a positive semidefinite matrix. In the special case $Q=I$ (the $q\times q$ 
identity matrix), in \cite{Wittwar2018} it is shown that the native space of $K$ is the tensor product of $q$ copies of the native space of $\tilde K$, i.e.,
\begin{align*}
\calh_K(\Omega)= \left\{f:\Omega\to\R^q : f_j\in \calh_{\tilde K}(\Omega), 1\leq j\leq q\right\}
\end{align*}
with
\begin{align*}
\inner{\calh_K}{f, g} = \sum_{j=1}^q \inner{\calh_{\tilde K}}{f_j, g_j}. 
\end{align*}
This simplification will give convenient advantages when implementing some of the methods discussed in Section \ref{sec:interp}.

\section{Data based surrogates}\label{sec:surrogates_general}
We can now introduce in general terms the two surrogate modeling techniques that we will discuss, namely (regularized) kernel interpolation and \gls*{svr}.

For both of them, the idea is to represent the expensive map to be reduced as a function $f:\Omega\to \R^q$ that maps an input $x\in\Omega$ to an output $y\in \R^q$. 
Here $f$ is assumed to be only continuous, and the set $\Omega$ can be arbitrary as long as a positive definite kernel $K$ can be defined on it. Moreover, 
the function does not need to be known in any particular way except than through its evaluations on a finite set $X_n:=\{x_k\}_{k=1}^n\subset\Omega$ of pairwise distinct 
data points, resulting in data values $Y_n:=\{y_k:=f(x_k)\}_{k=1}^n\subset\R^q$.

The goal is to construct a function $s\in\calh$ such that $s(x)$ is a good approximation of $f(x)$ for all $x\in\Omega$ (and not only for $x\in X_n$), 
while being significantly faster to evaluate. The process of computing $s$ from the data $(X_n, Y_n)$ is often referred to as training of the surrogate $s$, and the 
set $(X_n, Y_n)$ is thus called training dataset. 

The computation of the particular surrogate is realized as the solution of an infinite dimensional optimization problem. In general terms, we define a loss function 
\begin{align*}
L: \calh\times \Omega^n\times\left(\R^q\right)^ n\to \R_{\geq 0} \cup\{+\infty\},
\end{align*}
which takes as input a candidate surrogate $g\in\calh$ and the values $X_n\in\Omega^n$, $Y_n\in\left(\R^q\right)^ n$, and returns a measure of the data-accuracy of 
$g$. 
Then, the surrogate $s$ is defined as a minimizer, if it exists, of the cost function 
\begin{align*}
J(g):= L(g, X_n, Y_n) + \lambda \|g\|_{\calh}^2,
\end{align*}
where the second part of $J$ is a regularization term that penalizes solutions with large norm. The tradeoff between the data-accuracy term and the regularization term 
is controlled by the regularization parameter $\lambda\geq0$. 

For the sake of presentation, we restrict in the remaining of this section to the case of scalar-valued functions, i.e., $q=1$. The general case follows by using 
matrix 
valued kernels as introduced in Section \ref{sec:mat_val}, and the corresponding definition of orthogonal projections.

The following fundamental Representer Theorem characterizes exactly some solutions of this problem, and it proves that the surrogate will be a function 
\begin{align*}
s\in V(X_n):=\Sp{K(\cdot, x_i), x_i\in X_n}
\end{align*}
i.e., a finite linear combination of kernel translates on the training points. A first version of this result was proven in \cite{Wahba1970}, while we refer to 
\cite{Schoelkopf2001v} for a more general statement.

\begin{theorem}[Representer Theorem]\label{th:repre_theorem}
Let $\Omega$ be a nonempty set, $K$ a PD kernel on $\Omega$, $\lambda>0$ a regularization parameter, and let $(X_n, Y_n)$ be a training set.
Assume that $L(s, X_n, Y_n)$ depends on $s$ only via the values $s(x_i)$, $x_i\in X_n$.

Then, if the optimization problem
\begin{equation}\label{eq:minimization_problem}
\argmin_{g\in\calh }J(g) := L(g, X_n, Y_n) + \lambda \|g\|_{\calh}^2
\end{equation}
has a solution, it has in particular a solution of the form
\begin{align}\label{eq:regularized_s}
s(x) := \sum_{j=1}^n \alpha_j K(x, x_j),\;\;x \in\Omega,
\end{align}
for suitable coefficients $\alpha\in\R^n$.
\end{theorem}

\begin{proof}
We prove that for every $g\in\calh$ there exists $s\in V(X_n)$ such that $J(s)\leq J(g)$.
To see this, we decompose $g\in\calh$ as
\begin{align*}
g = s + s^{\perp},\;s\in V(X_n),\ s^{\perp}\in V(X_n) ^{\perp}.
\end{align*}
In particular, since $K(\cdot, x_i)\in V(X_n)$, we have by the reproducing property of the kernel that
\begin{align*}
s^{\perp}(x_i) = \inner{\calh}{s, K(\cdot, x_i)} = 0,\;\;1\leq i\leq n,
\end{align*}
thus $g(x_i) = s(x_i) + s^{\perp}(x_i) = s(x_i)$ for $1\leq i\leq n$, and it follows that $L(g, X_n, Y_n) = L(s, X_n, Y_n)$.
Moreover, again by orthogonal projection we have $\|g\|_{\calh}^2 = \|s\|_{\calh}^2 + \|s^{\perp}\|_{\calh}^2$.
Since $\lambda\geq 0$, we then obtain
\begin{align*}
J(s) &= L(s, X_n, Y_n) + \lambda \|s\|_{\calh}^2 = L(g, X_n, Y_n) + \lambda \|s\|_{\calh}^2 \\
& = L(g, X_n, Y_n) + \lambda \|g\|_{\calh}^2 - \lambda \|s^{\perp}\|_{\calh}^2= J(g) - \lambda \|s^{\perp}\|_{\calh}^2 \leq J(g).
\end{align*}
Thus, if $g\in\calh$ is a solution then $s\in V(X_n)$ is also a solution.
\end{proof}

The existence of a solution will be guaranteed by choosing a convex cost function $J$, i.e., since the regularization term is always convex, by choosing 
a convex loss function. Then the theorem states that solutions of the infinite dimensional optimization problem can be computed by solving a finite dimensional 
convex one.

This is a great result, but observe that the evaluation of $s(x)$, $x\in\Omega$, requires the evaluation of the $n$-terms linear combination \eqref{eq:regularized_s}, 
where $n$ is the size of the dataset. Assuming that the kernel can be evaluated in constant time, the 
complexity of this operation is $\mathcal O(n)$.  
Thus, to achieve the promised speedup in evaluating the surrogate in place of the function $f$, we will consider in the following methods that enforce sparsity 
in $s$, i.e., which compute approximate solution where most of the coefficients $\alpha_j$ are zero. If the nonzero coefficients correspond to an index set 
$I_N:=\{i_1,\dots, i_N\}\subset 
\{1,\dots,n\}$, the complexity is reduced to $\mathcal O(N)$. 

Taking into account this sparsity and denoting $X_N:=\{x_i\in X_n: i\in I_N\}$ and $\alpha:=\left(\alpha_i: i\in I_N\right)$, we can summarize in Algorithm 
\ref{alg:online} the online phase for any of the following 
algorithms, consisting in the evaluation of $s$ on a set of points $X_{te}\subset\Omega$. Here and in the following, we denote as $s(X):=(s(x_1), \dots, s(x_m))^T\in 
\R^{m}$ the vector of evaluations of $s$ on a set of points $X:=\{x_i\}_{i=1}^m\subset\Omega$.

\begin{algorithm}[h!t]
  \caption{Kernel surrogate - Online phase}
  \begin{algorithmic}[1]
      \State{Input: $X_{N}\in\Omega^{N},  \alpha\in\R^{N}$, kernel $K$ (and kernel parameters), test points 
$X_{te}:=\{x^{te}_i\}_{i=1}^{n_{te}}\in\Omega^{n_{te}}$}
        \State{Compute the kernel matrix $A_{te}\in \R^{n_{te}\times N}$, $(A_{te})_{ij}:= K(x^{te}_i, x_{i_j})$.}
	\State{Evaluate the surrogate $s(X^{te}) = A_{te} \alpha$.}
     \State{Output: evaluation of the surrogate $s(X^{te})\in\R^{n_{te}}$.}
    \end{algorithmic}\label{alg:online}
\end{algorithm}

\begin{remark}[Normalization of the cost function]
It is sometimes convenient to weight the loss term in the cost function \eqref{eq:minimization_problem} by a factor $1/n$, which normalizes its value with respect to the 
number of data. We do not use this convention here, and we only remark that this is equivalent to the use of a regularization parameter $\lambda = n \lambda'$ for 
a given $\lambda'>0$. 
\end{remark}

\section{Kernel interpolation}\label{sec:interp}
The first method that we discuss is (regularized) kernel interpolation. In this case, we consider the square loss function 
\begin{align*}
L(s, X_n, Y_n):= \sum_{i=1}^n \left(s(x_i) - y_i\right)^2, 
\end{align*}
which measures the pointwise distance between the surrogate and the target data. 
We have then the following special case of the Representer Theorem. We denote as $y\in\R^n$ the vector of output data, assuming again for now that $q=1$.
\begin{cor}[Regularized interpolant]\label{th:repre_theorem_reg_interp}
Let $\Omega$ be a nonempty set, $K$ a PD kernel on $\Omega$, $\lambda\geq0$ a regularization parameter. 
For any training set $(X_n, Y_n)$ there exists an approximant of the form
\begin{align}\label{eq:regularized_interpolant}
s(x) = \sum_{j=1}^n \alpha_j K(x, x_j),\;\;x \in\Omega,
\end{align}
where the vector of coefficients $\alpha\in\R^n$ is a solution of the linear system
\begin{equation}\label{eq:regularized_interpolant_A}
(A + \lambda I) \alpha = y,
\end{equation}
where $A\in\R^{n\times n}$, $A_{ij}:=K(x_i, x_j)$, is the kernel matrix on $X_n$.
Moreover, if $K$ is SPD this is the unique solution of the minimization problem \eqref{eq:minimization_problem}.
\end{cor}
\begin{proof}
The loss $L$ is clearly convex, so there exists a solution of the optimization problem, and by Theorem \ref{th:repre_theorem} we know that we can restrict to 
solutions 
in $V(X_n)$.

We then consider functions $s:= \sum_{j=1}^n \alpha_j K(\cdot, x_j)$ for some unknown $\alpha\in\R^n$. Computing the inner product as in Proposition 
\ref{prop:two_rkhs}, we obtain 
\begin{align*}
s(x_i) =  \sum_{j=1}^n \alpha_j K(x_i, x_j) = (A \alpha)_i, \;\; \|s\|_{\calh}^2 = \sum_{i,j=1}^n \alpha_i\alpha_j K(x_i, x_j) = \alpha^T A \alpha.
\end{align*}
The functional $J$ restricted to $V(X_n)$ can be parametrized by $\alpha\in\R^n$, and thus it can be rewritten as $\tilde J:\R^n\to\R$ with  
\begin{align*}
\tilde J(\alpha) &= \|A \alpha - y\|_2^2  + \lambda \alpha^T A \alpha = (A \alpha - y)^T (A \alpha - y) + \lambda\alpha^T A \alpha\\
&= \alpha^T A^T A \alpha - 2 \alpha^T A^T y + y^T y + \lambda\alpha^T A \alpha,
\end{align*}
which is convex in $\alpha$ since $A$ is positive semidefinite. Since $A$ is symmetric, its gradient is
\begin{align*}
\nabla_{\alpha} \tilde J(\alpha) &= 2 A^T A \alpha - 2 A^T y + 2 \lambda A \alpha = 2 A (A \alpha - y + \lambda \alpha),
\end{align*}
i.e., $\nabla_{\alpha} \tilde J(\alpha) = 0$ if and only if $A\left(A   + \lambda I \right)\alpha = A y$, which is satisfied by $\alpha$ such that $\left(A   + 
\lambda I 
\right)\alpha = y$. If $K$ is SPD then both $A$ and $A+\lambda I$ are invertible, so this is the only solution.
\end{proof}

The extension to vector valued functions, i.e. $q>1$, is straightforward using the separable matrix valued kernels with $B=I$ of Section \ref{sec:mat_val}. Indeed, in 
this case the data values are vectors $y_i:=f(x_i)\in\R^q$, and thus in the interpolant 
\eqref{eq:regularized_interpolant} also the coefficients are vectors $\alpha_j\in\R^q$. The linear system \eqref{eq:regularized_interpolant_A} has the same 
matrix, but 
instead $\alpha, y\in \R^{n\times q}$ are defined as
\begin{align}\label{eq:mat_val_coeffs}
\alpha := \left(\alpha_1, \dots, \alpha_n\right)^T, \;\;y := \left(y_1, \dots, y_n\right)^T.
\end{align}
We remark that in the following the values $x_i$, $y_i$, $s(x)$, and $\alpha_k$ have always to be understood as row vectors when $q>1$. 
This notation is very convenient when representing the coefficients as the 
solution of a linear system. Furthermore, the representation of the dataset samples $(x, y)$ is quite natural when dealing with tabular data, where each column 
represents a feature and each row a sample vector.

For $K$ \gls*{spd} and pairwise distinct sample locations $X_n$ we can also set $\lambda:=0$ and obtain pure interpolation, i.e., the solution satisfies $L(s, X_n, 
Y_n) = 0$, or
\begin{align*}
s(x_i) = y_i, \;\; 1\leq i\leq n.
\end{align*}
Observe that this means that with this method we can exactly interpolate arbitrary continuous functions on arbitrary pairwise distinct scattered data in any 
dimension, as opposite to many other techniques which require complicated conditions on the interpolation points or a grid structure.
Moreover, this approximation process has several optimality properties in $\calh$, which remind of similar properties of spline interpolation. 
\begin{proposition}[Optimality of kernel interpolation]
Let $K$ be \gls*{spd}, $f\in\calh$, and $\lambda=0$. Then $s$ is the orthogonal projection of $f$ in $V(X_n)$, and in particular
\begin{align*}
\|f - s\|_{\calh} = \min_{g\in V(X_n)} \|f - g\|_{\calh}.
\end{align*}
Moreover, if $S : = \left\{g\in\calh : g(x_i) = f(x_i), 1\leq i\leq n\right\}$, then
\begin{align*}
\left\|s\right\|_{\calh} = \min_{g\in S} \|g\|_{\calh},
\end{align*}
i.e., $s$ is the minimal norm interpolant of $f$ on $X_n$.
\end{proposition}
\begin{proof}
The proof is analogous to the proof of the Representer Theorem, using a decomposition $f = g + g^{\perp}$, and proving that $s=g$.
\end{proof}

We will see in Section \ref{sec:parameter_selection} a general technique to tune $\lambda$ using the data, which should return $\lambda = 0$ (or very small) 
when this is the best option. Nevertheless, also for an \gls*{spd} kernel there are at least two 
reasons to still consider regularized interpolation. First, the data can be affected by noise, and thus an exact pointwise recovery does not make much sense. 
Second, 
a positive parameter $\lambda> 0$ improves the condition number of the linear system, and thus the stability of the solution. Indeed, the $2$-condition number of $A + 
\lambda I$ is
$$
\kappa(\lambda):=\frac{\lambda_{\max}(A + \lambda I)}{\lambda_{\min}(A + \lambda I)} = \frac{\lambda_{\max}(A) + \lambda}{\lambda_{\min}(A) + \lambda }, 
$$
which is a strictly decreasing function of $\lambda$, with $\kappa(0) = \kappa(A)$ and $\lim\limits_{\lambda\to\infty} \kappa(\lambda) = 1$. Moreover (see 
\cite{WendlandRieger2005}) this increased stability can be achieved by still controlling the pointwise accuracy. Namely, if $f\in\calh$, it holds 
\begin{align*}
\left\|y_i - s(x_i)\right\|_2\leq \sqrt{\lambda}\|f\|_{\calh}\;\; 1\leq i\leq n. 
\end{align*}

We can then summarize the offline phase for regularized kernel interpolation in Algorithm \ref{alg:interpolation_offline}.

\begin{algorithm}[h!t]
  \caption{Regularized Kernel interpolation - Offline phase}
  \begin{algorithmic}[1]
      \State{Input: training set $X_{n}\in\Omega^{n}$, $Y_{n}\in\left(\R^q\right)^n$, kernel $K$ (and kernel parameters), regularization parameter $\lambda\geq 
0$.}
        \State{Compute the kernel matrix $A\in \R^{n\times n}$, $A_{ij}:= K(x_i, x_j)$.}
        \State{Solve the linear system $(A + \lambda I) \alpha = y$.}
     \State{Output: coefficients $\alpha\in\R^{n\times q}$.}
    \end{algorithmic}\label{alg:interpolation_offline}
\end{algorithm}

\begin{remark}[Flat limit]
The matrix $A$ can be seriously ill-conditioned for certain kernels, and this constitutes a problem at least in the case of pure interpolation. It can also be 
proven 
that kernels which guarantee a faster error convergence result in a worse conditioned matrix \cite{Schaback1995}.

For \gls*{rbf} kernels, this happens especially for $\gamma\to 0$ (the so called flat limit), and it is usually not a good idea to directly solve the linear 
system. In the last years there has been very active research to compute $s$ via different formulations, which rely on different representations of the 
kernel. We mention here mainly the RBF-QR algorithm\footnote{\url{http://www.it.uu.se/research/scientific_computing/software/rbf_qr}} \cite{Larsson2013a,Fornberg2011}
and the Hilbert-Schmidt SVD\footnote{\url{http://math.iit.edu/~mccomic/gaussqr/}} \cite{Fasshauer2012b} . 
 Both methods are limited so far to only some kernels, but they manage to achieve a great accuracy, which is usually impossible to obtain with the direct 
solution of the 
linear system.   
\end{remark}

\begin{remark}[Error estimation]
For \gls*{spd} translational invariant kernels there is a very detailed error analysis of the interpolation process ($\lambda = 0$), for which we refer to  
\cite[Chapter 11]{Wendland2005}. We only mention that these error bounds assume that $f\in\calh$, and are of the form
\begin{align*}
\left\|f - s\right\|_{L_{\infty}(\Omega)} \leq C h_{n}^p \|f\|_{\calh},
\end{align*}
where $C>0$ is a constant independent of $f$, and $h_{n}$ is the fill distance of $X_n$ in $\Omega$, i.e., 
$$
h_n:=h_{X_n,\Omega} := \sup\limits_{x\in\Omega}\min\limits_{x_j\in X_n}\|x-x_j\|,
$$
which is the analogue of the mesh width for scattered data. Moreover, the order of convergence $p>0$ is dependent on the smoothness of the kernel.
In particular, these error bounds can be proven to be optimal when the native space of $K$ is a Sobolev space.

Moreover, these results have been recently extended to the case of regularized interpolation ($\lambda>0$) in \cite{WendlandRieger2005,Rieger2008b}.
\end{remark}

\subsection{Kernel greedy approximation}\label{sec:greedy}
The surrogate constructed via Corollary \ref{th:repre_theorem_reg_interp} involves a linear combination of $n$ 
terms, where $n$ is the size of the dataset. In general, there is no reason to assume that the result has any sparsity, i.e., in general all the $\alpha_j$ will be 
nonzero, and it is thus necessary to introduce some technique to enforce this sparsity.

A very effective way to achieve this result is via greedy algorithms. The idea is to select a small subset $X_N\subset X_n$, $N\ll n$, given by indexes $I_N\subset\{1, 
\dots, n\}$, and to solve the corresponding restricted problem with the dataset $(X_N, Y_N)$ to compute a surrogate
\begin{align}\label{eq:greedy_surrogate}
s_N(x):= \sum_{k \in I_N} \alpha_k K(x, x_k),
\end{align}
where the coefficient vectors are computed based on \eqref{eq:regularized_interpolant_A}, and are in general different from the ones of the full surrogate. If we 
manage to select $I_N$ in a proper way, we will obtain $s_N(x)\approx f(x)$ for all $x\in\Omega$, while the evaluation of $s_N(x)$ is now only of order $\mathcal O(N)$.

An optimal selection of $X_N$ is a combinatorial problem and thus is very expensive and in practice computationally intractable. The idea of greedy algorithms is instead 
to perform this selection incrementally, i.e., adding at each iteration only the most promising new point, based on some error indicator. 

The general structure of the algorithm is described in Algorithm \ref{alg:greedy_general}. 
For the moment, we consider a generic selection rule $\eta:X_n\times\N\times \Omega^n\times\left(\R^q\right)^{n}\to\R_{\geq 0}$ that selects points based on the value 
$\eta\left(x,N, 
X_n, Y_n\right)$. This is a compact notation to denote that the selection rule assigns a score to a point $x\in \Omega$, and it is computed 
using various quantities that depend on the dataset $(X_n, Y_n)$ and on the iteration number $N$, including in particular the surrogate computed at the previous 
iteration.
The algorithm is terminated by means of a given tolerance $\tau>0$. 

\begin{algorithm}[h!t]
  \caption{Kernel greedy approximation - Offline phase}
  \begin{algorithmic}[1]
      \State{Input: training set $X_{n}\in\Omega^{n}$, $Y_{n}\in\left(\R^{q}\right)^n$, kernel $K$ (and kernel parameters), regularization parameter $\lambda\geq 
0$, selection rule $\eta$, tolerance $\tau$.}
      \State{Set $N:=0$, $X_0:=\emptyset$, $V(X_0):=\{0\}$, $s_0:=0$.}
       \Repeat
	\State{Set $N:= N + 1$}
        \State{Select $x_N := \argmax\limits_{x\in X_n\setminus X_{N-1}}\ \eta(x, N, X_n, Y_n)$.}\label{line:alg_greedy_approx}
	\State{Define $X_N:= X_{N-1}\cup\{x_N\}$ and $V(X_N):=\Sp{K(\cdot, x_i), x_i\in X_N}$}
	\State{Compute the surrogate $s_N$ with dataset $(X_N, Y_N)$ with \eqref{eq:regularized_interpolant_A}.}\label{line}
    \Until{$\eta\left(x_N, N,  X_n, Y_n\right)\leq \tau $}
     \State{Output: surrogate $s_N$ (i.e. coefficients $\alpha\in\R^{N\times q}$).}
    \end{algorithmic}\label{alg:greedy_general}
\end{algorithm}

\begin{remark}
In the case that the maximizer of $\eta$ the line \ref{line:alg_greedy_approx} of Algorithm \ref{alg:greedy_general}  is not unique, only one of the multiple points is 
selected and included in $X_N$.
\end{remark}

In line \ref{line} of the algorithm, we need to compute the surrogate $s_N$ with dataset $(X_N, Y_N)$. This step can be highly simplified by reusing $s_{N-1}$ as much as 
possible, thus improving the efficiency of the algorithm. As a side effect, with this incremental procedure it is easy to update the surrogate if the accuracy has 
to be improved.

This can be achieved using the Newton basis, which is defined in analogy to the Newton basis for polynomial interpolation. It has been introduced in 
\cite{Muller2009,Pazouki2011} for $K$ SPD, and extended to the case of $K$ PD and $\lambda>0$ in \cite{SaWiHa2018}, and we refer to these papers for the proof of the 
following result.

\begin{proposition}[Newton basis]\label{def:newton}
Let $\Omega$ be non empty, $\lambda\geq 0$, $K$ be PD on $\Omega$ or SPD when $\lambda=0$. Let $X_n\subset\Omega$ be pairwise distinct, and let $I_N\subset \{1, 
\dots, n\}$. 
Let moreover $K_{\lambda}(x, y):=K(x, y)+\lambda \delta_{y}(x)$ for all $x, y\in\Omega$, and denote its RKHS as $\calh_{\lambda}$.

The Newton basis $\{v_j\}_{j=1}^N$ is defined as the Gram-Schmidt 
orthonormalization  of $\{K_{\lambda}(\cdot, x_i)\}_{i\in I_N}$ in $\calh$, i.e.,
\begin{align*}
v_1(x) &: = \frac{K_{\lambda}(x, x_{i_1})}{\norm{\calh_{\lambda}}{K_{\lambda}(\cdot, x_{i_1})}}=\frac{K_{\lambda}(x, x_{i_1})}{\sqrt{K_{\lambda}(x_{i_1}, x_{i_1})}}\\
\tilde v_k(x) &: =K_{\lambda}(x, x_{i_k}) - \sum_{j=1}^{k-1} 
v_j(x_{i_k}) v_j(x)\\
v_k(x) &: = \frac{\tilde v_k(x)}{\|\tilde v_k\|_{\calh_{\lambda}}}=\frac{\tilde v_k(x)}{\sqrt{\tilde v_k(x_k)}},\;\; 1<k\leq N. 
\end{align*}
Moreover, for all $1\leq k\leq N$, we have
\begin{align*}
v_k(x) = \sum_{j=1}^N \beta_{jk} K_{\lambda}(x, x_{i_j}),
\end{align*}
and, if $B\in\R^{N\times N}$, $B_{jk}:= \beta_{jk}$, and $V\in \R^{N\times N}$, $V_{jk} := v_k(x_j)$, then $B, V$ are triangular, $B = V^{-T}$, and
\begin{align*}
A_N + \lambda I = V V^{T}
\end{align*}
is the Cholesky decomposition of the regularized kernel matrix $A_N+\lambda I\in\R^{N\times N}$, $A_{jk}:= K(x_{i_j}, x_{i_k})$, with pivoting given by $I_N$.
\end{proposition}

Observe that this basis is nested, i.e., we can incrementally add a new element without recomputing the previous ones. Even more, with this basis the surrogate can be 
computed as follows.

\begin{proposition}[Incremental regularized interpolation]\label{prop:newton_interpolation}
Let $\Omega$ be non empty, $\lambda\geq 0$, $K$ be PD on $\Omega$ or SPD when $\lambda=0$. Let $(X_N, Y_N)$ be the subset of $(X_n, Y_n)$ corresponding to indexes $I_N$, 
for all $N\leq n$.

Let $\tilde s_0:=0$, and, for $N\geq 1$, compute the following incremental function
\begin{align}\label{eq:update_interp}
\tilde s_N(x) = \sum_{k=1}^N c_k v_k(x) = c_N v_N(x) + s_{N-1}(x),\;\; c_N:= \frac{y_{i_N} - \tilde s_N(x_{i_N})}{\sqrt{v_N(x_{i_N})}}.
\end{align}
Then, for all $N$, the regularized interpolant can be computed as  
\begin{align*}
s_N(x) = \sum_{j=1}^N \alpha_j K(x, x_{i_j})\;\;\text{ where }\;\; \alpha := V^{-T} c. 
\end{align*}
\end{proposition}

\begin{remark}
In the case $\lambda = 0$ and $K$ SPD, the function $\tilde s_N$ coincides with the interpolant $s_N$. We refer to \cite{Pazouki2011,SaWiHa2018} for the details. 
\end{remark}

We are now left to define the selection rules, represented by $\eta$, to select the new point at each iteration. 

For this, we first need to define the power function, which gives an upper bound on the interpolation error, and it can be defined using the Newton basis as
\begin{align}\label{eq:update_pf}
P_N(x)^2:= K_{\lambda}(x, x) - \sum_{j=1}^N v_j(x)^2. 
\end{align}
Its relevance is due to the fact that it provides an upper bound on the pointwise (regularized) interpolation error, i.e., if $x\notin X_n$, and $K$ is PD, 
or SPD when $\lambda=0$, it holds for all $f\in\calh$ that 
\begin{align}\label{eq:upper_bound_pf}
\left|f(x) - s_N(x)\right|\leq P_n(x) \norm{\calh}{f}.
\end{align}
This function is well known and studied in the case of pure interpolation (see e.g. \cite[Chapter 11]{Wendland2005}), for which the upper bound holds for all 
$x\in\Omega$, and it can be easily extended to the case of regularized interpolation (see \cite{SaWiHa2018}). In both cases, it can be proven that $P_n(x) = 0$ if and 
only if $x\in X_n$, and its maximum is strictly decreasing with $N$.

\begin{remark}
This interpolation technique is strictly related to the Kriging method and to Gaussian Process Regression (see e.g. \cite{Olea2012, Rasmussen2006}. In 
this case the kernel represents the covariance kernel 
of the prior distribution, and the power function is the Kriging  variance, or variance of the posterior distribution (see \cite{scheuerer2013}).
\end{remark}

We can then define the following selection rules. We assume to have a dataset $(X_n, Y_n)$, and to have already selected $N$ points corresponding to indexes $I_{N-1}$. 
We use the notation $[1,n]:=\{1, \dots, n\}$, and we have
\begin{itemize}
 \item {$P$-greedy:} \hspace{1cm}$i_N:= \argmax\limits_{i\in [1,n] \setminus I_{N-1}} P_{N-1}(x_i)$
 \item {$f$-greedy:} \hspace{1cm}$i_N:= \argmax\limits_{i\in [1,n] \setminus I_{N-1}} \left|y_i - s_{N-1}(x_i)\right|$
 \item {$f/P$-greedy:} \hspace{.5cm}$i_N:= \argmax\limits_{i\in [1,n] \setminus I_{N-1}} \frac{\left|y_i - s_{N-1}(x_i)\right|}{P_{N-1}(x_i)}$.
\end{itemize}

Observe that all the selections are well defined, since $P_{N-1}(x_i)\neq 0$ for all $i\notin I_{N-1}$ if $X_N$ are pairwise distinct, and they can be efficiently 
implemented 
by using the update rules \eqref{eq:update_interp} for $s_N$ and \eqref{eq:update_pf} for $P_N$. Moreover, they are motivated by different ideas: 
The $P$-greedy selection tries to minimize the Power function, thus providing a uniform upper bound on the error for any function $f\in\calh$ via 
\eqref{eq:upper_bound_pf}; the $f$- and $f/P$-greedy (which reads ``$f$-over-$P$-greedy''), on the other hand, use also the output data, and produce points which are 
suitable to approximate a single 
function and thus are expected to result in a better approximation. In the case of $f$-greedy this is done by including in the set of points the location where the 
current largest error is achieved, thus reducing the maximum 
error. The $f/P$-greedy selection, instead, reduces the error in the $\calh$-norm, and indeed it can be proven to be locally optimal, i.e., it 
guarantees the maximal possible reduction of the error, in the $\calh$-norm, at each iteration.

We can now describe the full computation of the greedy regularized interpolant in Algorithm \ref{alg:greedy_newton}. It 
realizes the computation of the sparse surrogate $s_N$ by selecting the points $X_N$ via the index set $I_N$, and computing only the nonzero coefficients $\alpha$.
Moreover, using the nested structure of the Newton basis and the incremental computation of Proposition \ref{prop:newton_interpolation}, the algorithm needs only to 
compute the columns of the full kernel matrix corresponding to the index set $I_N$, and thus there is no need to compute nor store the full $n\times n$ matrix, i.e., the 
implementation is matrix-free. In addition, again using Proposition \ref{prop:newton_interpolation} most of the operations are done in-place, i.e., some vectors are used 
to store and update the values of the Power 
Function and of $y$. In the algorithm, we use a Matlab-like notation, i.e., $A(I_N, :)$ denotes the submatrix of $A$ consisting of rows $I_N$ and of all the columns. 
Moreover, the notation $v^2$ denotes the pointwise squaring of the entries of the vector $v$.

\begin{algorithm}[h!t]
  \caption{Kernel greedy approximation - Offline phase}
  \begin{algorithmic}[1]
      \State{Input: training set $X_{n}\in\Omega^{n}$, $Y_{n}\in\left(\R^{q}\right)^{n}$, kernel $K$ (and kernel parameters), regularization parameter $\lambda\geq 
0$, selection rule $\eta$, tolerance $\tau$.}
      \State{Set $N:=0$, $I_0:=\emptyset$, $V:=[\cdot ]\in\R^{n\times 0}$, $p:=\mbox{diag}(K(X_n, X_n))\in\R^n$}
       \Repeat
        \State{Set $N=N+1$}
        \State{Select $i_N := \argmax\limits_{i\in [1,n]\setminus I_{N-1}}\ \eta(x_i, N, X_n, Y_n)$.}
	\State{Generate column $v := K(X_n, x_{i_N})$}
	\State{Project $v:= v - V  V(i_N, :)^T$}
	\State{Normalize $v = v / \sqrt{v(i_N)}$}
	\State{Compute $c_N:=y(i_N)/\sqrt{v(i_N)}$}
	\State{Update the power function $p:= p - v^2$}
	\State{Update the residual $y:= y - c_N  v$}
	\State{Update the inverse $C = V(I_N, :)^{-1}$}
	\State{Add the column $V = [V, v_N]$}
	\State{Add the coefficient $c = [c^T, c_N]^T$}\label{linec}
	\State{Update $I_N:=I_{N-1}\cup \{i_N\}$}
    \Until{$\eta\left(x_{N}, N, X_n, Y_n\right)\leq\tau $}
     \State{Set $\alpha = C c$}
     \State{Output: $\alpha\in\R^{N\times q}$, $I_N$.}
    \end{algorithmic}\label{alg:greedy_newton}
\end{algorithm}
The set of points $X_N$ defined by $I_N$, and the coefficients $\alpha$, can then be used in the online phase of Algorithm \ref{alg:online}.

\begin{remark}[Vector valued functions and implementation details]
Algorithm \ref{alg:greedy_newton} and the overall procedure are well defined for arbitrary $q\geq 1$. Indeed, using the separable matrix valued kernel of Section 
\ref{sec:mat_val}, the 
Newton basis only depends on the scalar valued kernel $K$, while the computation of the coefficients is valid by considering that now $c$, $\alpha$ are matrices 
instead of vectors. In particular, the computation of $c_N$ (line 14) and the update of $y$ (line 11) has to be done via column-wise multiplications. 

Moreover, observe that in line 12 we employ a standard technique to update the inverse of a lower triangular matrix, i.e., given $V_N\in\R^{N\times N}$ lower 
triangular with inverse $V_N^{-1}$, we define
\begin{align*}
V_{N+1}=\left[\begin{array}{cc}
        V_N & 0\\v^T&w 
        \end{array}
\right] 
\end{align*}
for $v\in\R^N$, $w\in\R$, and compute $V_{N+1}^{-1}$ by a simple row-update as
\begin{align*}
V_{N+1}^{-1}=\left[\begin{array}{cc}
        V_N^{-1} & 0\\- v^T V_N^{-1}/w&1/w
        \end{array}
\right].
\end{align*}

The present version of the algorithm for vector-valued functions has been introduced in \cite{Wirtz2013} and named Vectorial Kernel Orthogonal Greedy Algorithm (VKOGA). 
We keep the same abbreviation also for the regularized version, which has been studied in \cite{SaWiHa2018}.
\end{remark}

\begin{remark}[Convergence rates]
When the greedy algorithm is run by selecting points over $\Omega$ instead of $X_N$, there are also convergence rates for the resulting approximation processes. For pure 
interpolation (i.e., $K$ SPD, $\lambda = 0$) convergence of $f$-greedy has been proven in \cite{Mueller2009}, of $P$-greedy in \cite{SH16b}, and 
of $f/P$-greedy in 
\cite{Wirtz2013}, while in \cite{SaWiHa2018} the convergence rate of $P$-greedy has been extended to regularized interpolation. All the results make additional 
assumptions on the kernels, for which we refer to the cited literature. Nevertheless, we remark that the convergence rates for interpolation with $P$-greedy are 
quasi-optimal for translational invariant kernels, while the results for the other algorithms guarantee only a possibly significantly slower convergence rate. These 
results are believed to be significantly sub-optimal, since extensive experiments indicate that $f$- and $f/P$-greedy behave much better. This seems to suggest that 
there is space for a large improvement in 
the theoretical understanding of the methods. 
\end{remark}

\begin{remark}[Other techniques]
There are other techniques that can be applied to reduce the complexity of the evaluation of the surrogate $s$, which do not use greedy algorithms but instead
different approaches. First, there is a domain decomposition technique, known as Partition of Unity Method, which partitions $\Omega$ into subdomains, solves the 
(regularized) interpolation problem restricted to each patch, and then combines the results by a weighted sum of the local interpolants to obtain a global approximant. 
This 
method has the advantage that this offline phase can be completely parallelized. Moreover, when evaluating the surrogate only the few local interpolant having a support 
containing the test points have to be evaluated, thus requiring the evaluation of a few, small kernel expansions, thus providing a significant speed-up. The efficiency 
of this technique relies on an efficient search procedure to determine the local patches including the given points, which is the only limitation in the application to 
high dimensional problems. Both theoretical results and efficient implementations are available \cite{Wendland2002,Cavoretto2016h}.

Moreover, other sparsity-inducing techniques have been proposed, namely, the use of an $\ell_1$-regularization term \cite{Chen1998}, and the method of the Least Absolute 
Shrinkage and Selection Operator (LASSO) \cite{Tibshirani1996}.
\end{remark}

\section{Support Vector Regression}\label{sec:svr}
The second method that we present is Support Vector Regression (SVR) \cite{SS02}, which is based on different premises, but it still fits in the general 
framework of 
Section 
\ref{sec:surrogates_general}. In this case, we consider the $\varepsilon$-insensitive loss function 
\begin{align*}
L(s, X_n, Y_n):= \sum_{i=1}^n L_{\varepsilon}(s(x_i), y_i),\;\; L_{\varepsilon}(s(x_i), y_i): = \max(0, |s(x_i) - y_i| - \varepsilon),
\end{align*}
which is designed to linearly penalize functions $s$ which have values outside of an $\varepsilon$-tube around the data, while no distinction is made between 
function values that are inside this tube.

In this setting it is common to use the regularization parameter to scale the cost by a factor $1/\lambda$, and not 
the regularization term by a factor $\lambda$. The two choices are clearly equivalent, but we adopt here this different normalization to facilitate the comparison with 
the existing literature, and because this offers additional insights in the structure of the surrogate.

Since the problem is not quadratic (and not smooth), we first derive an equivalent formulation of the optimization problem \eqref{eq:minimization_problem}.
Assuming again that the output is scalar, i.e., $q=1$, the idea is to introduce non-negative slack variables $\xi^+,\xi^-\in\R^n$ which represent upper bounds on 
$L$ via
\begin{align}\label{eq:slack}
\xi_i^+&\geq \max(0, s(x_i) - y_i - \varepsilon),\;\;1\leq i\leq n,\\
\xi_i^-&\geq \max(0, y_i - s(x_i) - \varepsilon),\;\;1\leq i\leq n\nonumber,
\end{align}
and to minimize them in place of the original loss. 
With these new variables we can rewrite the optimization problem in the following equivalent way.

\begin{definition}[SVR - Primal form]\label{prop:svr_primal}
Let $\Omega$ be a nonempty set, $K$ a PD kernel on $\Omega$, $\lambda>0$ a regularization parameter. 
For a training set $(X_n, Y_n)$ the SVR approximant $\left(s, \xi^+, \xi^-\right)\in\calh\times\R^{2 n}$ is a solution of the quadratic optimization problem
\begin{align}\label{eq:svr_primal}
\min_{s\in\calh,\ \xi^+,\ \xi^-\in\R^n}& \frac{1}{\lambda} \onevec^T \left(\xi^+ + \xi^-\right) + \norm{\calh}{s}^2\\
\text{ s.t. } & s(x_i) - y_i - \varepsilon\leq \xi_i^+,\;\;1\leq i\leq n\nonumber\\
 -&s(x_i) +  y_i - \varepsilon\leq \xi_i^-,\;\;1\leq i\leq n\nonumber\\
 &\xi_i^+, \xi_i^-\geq 0,\;\;1\leq i\leq n\nonumber,
\end{align}
where $\onevec:= (1,\dots, 1)^T\in\R^n$. 
\end{definition}

For this rewriting of the optimization problem, we can now specialize the Representer Theorem as follows. 
\begin{cor}[SVR - Alternative primal form]\label{cor:svr_primal}
Let $\Omega$ be a nonempty set, $K$ a PD kernel on $\Omega$, $\lambda>0$ a regularization parameter. 
For any training set $(X_n, Y_n)$ there exists an SVR approximant of the form
\begin{align}\label{eq:svr_ansatz}
s(x) = \sum_{j=1}^n \alpha_j K(x, x_j),\;\;x \in\Omega,
\end{align}
where $\left(\alpha, \xi^+, \xi^-\right)\in\R^{3 n}$ is a solution of the quadratic optimization problem
\begin{align}\label{eq:svr_primal}
\min_{\alpha, \xi^+, \xi^-\in\R^n}& \frac{1}{\lambda} \onevec^T \left(\xi^+ + \xi^-\right) + \alpha^T A \alpha\\
\text{ s.t. } & (A \alpha)_i - y_i - \varepsilon\leq \xi_i^+,\;\;1\leq i\leq n\nonumber\\
 -&(A \alpha)_i +  y_i - \varepsilon\leq \xi_i^-,\;\;1\leq i\leq n\nonumber\\
 &\xi_i^+, \xi_i^-\geq 0,\;\;1\leq i\leq n\nonumber,
\end{align}
with $\onevec:= (1,\dots, 1)^T\in\R^n$, and $A\in\R^{n\times n}$, $A_{ij}:=K(x_i, x_j)$, the kernel matrix on $X_n$. Moreover, if $K$ is SPD this is the unique 
solution of the minimization problem \eqref{eq:minimization_problem}.
\end{cor}
\begin{proof}
The result is an immediate consequence of Proposition \ref{prop:svr_primal}, where we use the form \eqref{eq:svr_ansatz} for $s$ and compute 
its squared norm via Proposition \ref{prop:two_rkhs}.
\end{proof}

The slack variables \eqref{eq:slack} have a nice geometric interpretation. Indeed, the optimization process clearly tries to reduce their value as much as possible, 
while respecting the constraints. We state a more precise result in the following proposition, and give a schematic illustration in Figure \ref{fig:slack}.
\begin{proposition}[Slack variables]
Let $\alpha, \xi^+,\xi^-\in\R^n$ be a solution of \eqref{eq:svr_primal}, and let $s$ be the corresponding surrogate \eqref{eq:svr_ansatz}. 
Then, for each index $i\in\{1,\dots, n\}$, the values $\xi_i^+$, $\xi_i^-$ represent the distance of 
$s(x_i)$ from the $\varepsilon$-tube around $y_i$, and in particular
\begin{enumerate}
 \item If $s(x_i)>y_i+\varepsilon$, then $\xi_i^+>0$ and $\xi_i^-=0$
 \item If $s(x_i)<y_i-\varepsilon$, then $\xi_i^+=0$ and $\xi_i^->0$
 \item If $y_i-\varepsilon\leq s(x_i)\leq y_i+\varepsilon$, then $\xi_i^+=0$ and $\xi_i^-=0$.
\end{enumerate}
In particular, only one of $\xi_i^+$ and $\xi_i^-$ can be nonzero.
\end{proposition}

\begin{figure}
\begin{center}
\includegraphics[width=0.8\textwidth]{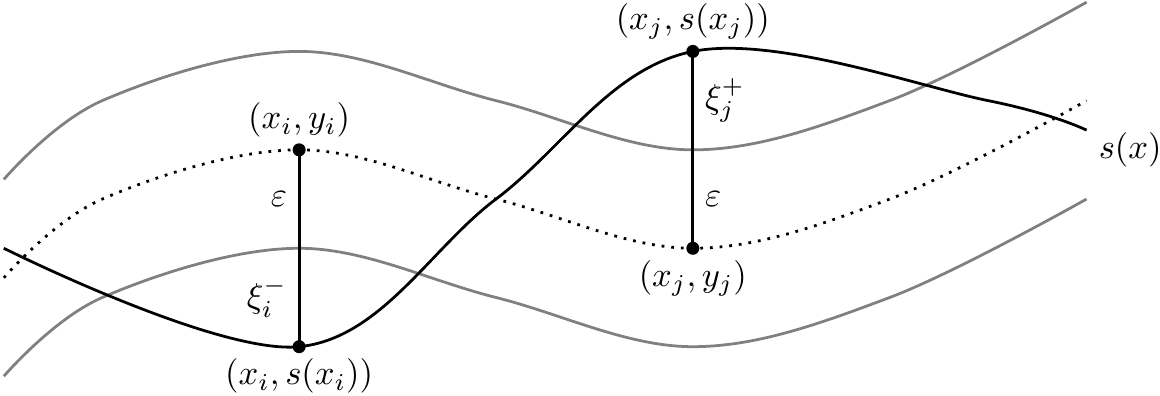} 
\end{center}
\caption{Illustration of the role of the slack variables in \eqref{eq:svr_primal}.}\label{fig:slack}
\end{figure}
Instead of solving the primal problem of Corollary \ref{cor:svr_primal}, it is more common to derive and solve the
following dual problem. 
Here again we denote as $y\in\R^n$ the vector of all scalar training target values.

\begin{proposition}[SVR - Dual form]\label{prop:svr_dual}
Let $\Omega$ be a nonempty set, $K$ a PD kernel on $\Omega$, $\lambda>0$ a regularization parameter. 
For any training set $(X_n, Y_n)$ there exists a solution $\left(\alpha^+, \alpha^-\right)\in \R^{2n}$ of the problem 
\begin{align}\label{eq:svr_dual}
\min\limits_{\alpha^+, \alpha^-\in \R^n}& \frac{1}{4} \left(\alpha^- - \alpha^+ \right)^T A\left(\alpha^- - \alpha^+\right) 
+ \varepsilon \onevec^T\left(\alpha^++\alpha^-\right)
+y^T\left(\alpha^+ - \alpha^-\right)\nonumber\\
\text{ s.t. } &\alpha^+, \alpha^-\in[0, 1/\lambda]^n,
\end{align}
which is unique if $K$ is SPD. Moreover, a solution of \eqref{eq:svr_primal} is given by 
\begin{align}\label{eq:svr_representation}
s(x) := \sum_{j=1}^n \frac{\alpha_j^--\alpha_j^+}{2} K(x, x_j),\;\;x \in\Omega,
\end{align}
with $\xi_i^+:= \max(0, s(x_i) - y_i - \varepsilon)$, $\xi_i^-:= \max(0, y_i - s(x_i) - \varepsilon)$.
\end{proposition}
\begin{proof}
We give a sketch of the proof, although a formal derivation requires more care, and we 
refer to \cite[Chapter 9]{SS02} for the details. 
The idea is to first derive the Lagrangian $\mathcal L:= \mathcal L(\alpha, \xi^+, \xi^-; \alpha^+, \alpha^-,\mu^+,\mu^-)$ for the primal problem \eqref{eq:svr_primal} 
using non-negative Lagrange multipliers $\alpha^+$, $\alpha^-$, $\mu^+$, $\mu^-\in\R^n$ for the inequality constraints, and then derive the dual problem by imposing the 
Karush-Kuhn-Tucker (KKT) conditions (see e.g. Chapter 6 in \cite{SS02}).

The Lagrangian is defined as
\begin{align}\label{eq:svr_lagrangian}
\mathcal L = &\ \frac1\lambda \onevec^T \left(\xi^+ + \xi^-\right) + \alpha^T A \alpha + (\mu^+)^T (-\xi^+) + (\mu^-)^T(-\xi^-)\\
&+\left(A \alpha - y -\varepsilon \onevec -\xi^+\right)^T \alpha^+ +\left(y - A \alpha -\varepsilon \onevec -\xi^-\right)^T\alpha^-\nonumber\\
=& \left(\alpha  + \alpha^+ - \alpha^-\right)^T A\alpha + \left(\frac1\lambda \onevec - \alpha^+-\mu^+\right)^T\xi^+ + \left(\frac1\lambda \onevec - \alpha^- - 
\mu^-\right)^T\xi^-\nonumber\\
&- \varepsilon \onevec^T\left(\alpha^++\alpha^-\right)- y^T\left(\alpha^+ - \alpha^-\right) \nonumber.
\end{align}
Using the symmetry of $A$, the partial derivatives of $\mathcal L$ with respect to the primal variables can be computed as
\begin{align}\label{eq:KKT_stat}
\nabla_\alpha \mathcal L  = 2 A\alpha + A (\alpha^+ - \alpha^-),\;\;
\nabla_{\xi^+} \mathcal L  = \frac1\lambda \onevec-\alpha^+ - \mu^+,\;\;
\nabla_{\xi_i^-} \mathcal L  = \frac1\lambda \onevec-\alpha^- - \mu^-,
\end{align}
and setting these three equalities to zero we obtain equations for $\alpha$, $\mu^+,\mu^-$, where in particular $\alpha = \frac{1}{2} \left(\alpha^- - \alpha^+\right)$ 
(which is the unique solution if $A$ is invertible).  
Substituting these values in the Lagrangian we get
\begin{align*}
\mathcal L&=\left(\alpha  + \alpha^+ - \alpha^-\right)^T A\alpha - \varepsilon \onevec^T\left(\alpha^++\alpha^-\right)- y^T\left(\alpha^+ - \alpha^-\right)\\
&=-\frac{1}{4} \left(\alpha^- - \alpha^+ \right)^T A\left(\alpha^- - \alpha^+\right) - \varepsilon \onevec^T\left(\alpha^++\alpha^-\right)- 
y^T\left(\alpha^+ - \alpha^-\right).
\end{align*}
The remaining conditions in \eqref{eq:svr_dual} stem from the requirements that the Lagrange multipliers are 
non-negative, and in particular $0\leq \mu_i^+ =  1/\lambda -\alpha^+_i$, i.e., $\alpha_i^+\leq 1/\lambda$, and similarly for $\alpha_i^-$.
\end{proof}

This dual formulation is particularly convenient to explain that the SVR surrogate has a built-in sparsity, i.e., the optimization process provides a solution where 
possibly many of the entries of $\alpha = \frac{1}{2}(\alpha^- - \alpha^+)$ are zero. This behavior is in strong contrast with the case of interpolation of Section 
\ref{sec:interp} where we needed to adopt special techniques to enforce this property. The points $x_i\in X_n$ with $\alpha_i\neq 0$ are called support 
vectors, 
which gives the name to the method. 

In particular, as for the slack variables there is a clean geometric description of this sparsity pattern, that 
gives additional insights into the solution. To see this we remark that, in addition to the stationarity KKT conditions \eqref{eq:KKT_stat}, an optimal solution 
satisfies also the complementarity KKT conditions
\begin{alignat}{2}
\label{eq:kkt_compl1}\alpha^+_i\left(s(x_i) - y_i - \varepsilon  - \xi_i^+ \right) &= 0,\quad \alpha^-_i\left(y_i - s(x_i) - \varepsilon  - \xi_i^- \right) &= 0\ \\
\label{eq:kkt_compl2}\xi^+_i\left(1/\lambda-\alpha_i^+\right) &= 0,\quad \xi^-_i\left(1/\lambda-\alpha^-_i\right) &= 0.
\end{alignat}
We then have the following:
\begin{enumerate}
 \item Equations \eqref{eq:kkt_compl1} state that $\alpha_i^+\neq 0$ only if $s(x_i) - y_i - \varepsilon  - \xi_i^+ =0$, and similarly for $\alpha_i^-$. Since 
$\xi_i^+\geq 0$, this happens only when $s(x_i) - y_i \geq \varepsilon$, i.e., only for points $(x_i, s(x_i))$ which are outside or on the boundary of the 
$\varepsilon$-tube.
\item In particular, if $\alpha_i^+\neq 0$ it follows that $s(x_i) - y_i \geq \varepsilon$, and thus $y_i - s(x_i) - \varepsilon -\xi_i^-\neq 0$, and then 
necessarily $\alpha_i^-=0$. Thus, at most one of $\alpha_i^+$ and $\alpha_i^-$ can be nonzero.
\item Equations \eqref{eq:kkt_compl2} imply that $\alpha_i^+, \alpha_i^- = 1/\lambda$ whenever $\xi_i^+,\xi_i^-$ is nonzero, i.e., whenever $s(x_i)$ is strictly outside 
of the $\varepsilon$-tube. The corresponding $x_i$ are called bounded support vector, and the value of the corresponding coefficients is indeed kept bounded by the value 
of the regularization 
parameter. Reducing $\lambda$, i.e., using less regularization, allows solutions with coefficients of larger magnitude.
\end{enumerate}
In summary, we can then expect that if $s$ is a good approximation of the data, it will be also a sparse approximation.

We summarize the offline phase for SVR in Algorithm \ref{alg:svr_offline}. We remark that in this case the extension to vector-valued functions is not as 
straightforward as for kernel interpolation, and it is thus common to train a separate SVR for each output component.

\begin{algorithm}[h!t]
  \caption{SVR - Offline phase}
  \begin{algorithmic}[1]
      \State{Input: training set $X_{n}\in\Omega^{n}$, $Y_{n}\in\R^{n}$, kernel $K$ (and kernel parameters), regularization parameter $\lambda\geq 0$, 
tube width $\varepsilon>0$.}
        \State{Compute the kernel matrix $A\in \R^{n\times n}$, $A_{ij}:= K(x_i, x_j)$.}
        \State{Solve the quadratic problem \eqref{eq:svr_dual}.}
        \State{Set $I_N:=\{i : \alpha^-_i\neq 0 \text{ or } \alpha^+_i\neq 0\}$.}
        \State{Set $\alpha_i:=(\alpha_i^--\alpha^+_i)/2$ for $i\in I_N$.}
        \State{Output: $\alpha\in\R^N$, $I_N$.}
    \end{algorithmic}\label{alg:svr_offline}
\end{algorithm}

\begin{remark}[General Support Vector Machines]
SVR is indeed one member of a vast collection of algorithms related to Support Vector Machines (SVMs). Standard SVMs  solve classification problems, 
i.e., $Y_n\subset\{0,1\}$. The original algorithm has been introduced as a linear algorithm (or, in the present understanding, as limited to the linear kernel, i.e., the 
polynomial kernel with $a=0$, $p=1$), and it has later been extended via the kernel trick to its general kernel version in \cite{Boser1992}. The SVR algorithms have 
instead been introduced in \cite{SS02}.

Moreover, the version presented here is usually called $\varepsilon$-\gls*{svr}. There exists 
also another non equivalent version called $\nu$-SVR, which adds another term in the cost function multiplied by a factor $\nu\in[0,1]$. This has 
the role of giving an upper bound on the number of support vectors and on the fraction of training data which are outside of the 
$\varepsilon$-tube (see Chapter 9 in \cite{SS02}).

We also remark that it is sometimes common to include in any SVM-based algorithm also an offset or bias term $b\in\R$, i.e., to obtain a surrogate $s(x) = \sum_{j=1}^n 
\alpha_j K(x, x_j) + b$. This changes in an obvious way the primal problem \eqref{eq:svr_primal}, while the dual contains also the constraint $\sum_{i=1}^n 
\left(\alpha_i^+ + \alpha_i^-\right)=0$. However, we stick here to this formulation and refer to \cite{Steinwart2011} for a discussion of statistical and numerical 
benefits of not using this offset term, at least in the case of \gls*{spd} kernels.
\end{remark}

\begin{remark}[Error estimation]
Also for SVR there is a detailed error theory, usually formulated in the framework of statistical learning theory (see \cite{Vapnik1998}).
Results are obtained by assuming that the dataset $(X_n, Y_n)$ is drawn from a certain unknown probability distribution, and then quantifying the approximation 
power of the surrogate. For a detailed treatment of this theory, we refer to \cite{SS02,Steinwart2008}.
Moreover, recently also deterministic error bounds for translational invariant kernels have been proven in \cite{Rieger2008b,Rieger2009b}.
\end{remark}

\subsection{Sequential Minimal Optimization}\label{sec:smo}
Although the optimization problem \eqref{eq:svr_dual} can in principle be solved with any quadratic optimization method, there exists a special algorithm, called 
Sequential Minimal Optimization (SMO) that is designed for SVMs and that performs possibly much better.

SMO is an iterative method which improves an initial feasible guess for $\alpha^+, \alpha^-\in\R^n$ until convergence, and the update is made such that the minimal 
possible number of entries of $\alpha$ are affected. In this way, very large problems can be efficiently solved. The original version of the algorithm has been 
introduced in 
\cite{Platt1998} for SVM, and it has later been adapted to more general methods such as SVR, that we use here to illustrate the structure of its 
implementation.  

The idea is to find at each iteration $\ell\in\N$ a  minimal set of indexes $I^{\ell}\subset\{1,\dots, n\}$ and optimize only the variables with indexes in $I^{\ell}$. 
The procedure is then iterated until the optimum is reached. If the SVR includes an offset term, as explained in the previous section we have constraints
\begin{align}\label{eq:smo_cond}
&\alpha_i^+,\alpha_i^-\in [0, 1/\lambda],\;\;1\leq i\leq n,\\
&\sum_{i=1}^n \left(\alpha_i^+ + \alpha_i^-\right)=0.\nonumber
\end{align}
Given a feasible solution $\left(\alpha_i^+, \alpha_i^-\right)^{(\ell)}$ at iteration $\ell\in\N$, it is thus not possible to update a single entry of 
$\alpha_i^+$ or $\alpha_i^-$ without violating the KKT conditions (since at most one between $\alpha_i^+$ and $\alpha_i^-$ need to be nonzero) or violating the second 
constraint. It is instead possible to select  two indexes $I^\ell:=\{i,j\}$ and in this case we have variables $\alpha_i^+,\alpha_i^-,\alpha_j^+,\alpha_j^-$ 
and we can 
solve the restricted quadratic optimization problem under the constraints
\begin{align*}
\alpha_i^+,\alpha_i^-\in [0, 1/\lambda], i\in I^{\ell},\;\;\;\; \sum_{i\in I^{\ell}} \left(\alpha_i^+ + \alpha_i^-\right)=R^{\ell}:=-\sum_{i\notin 
I^{\ell}}\left(\alpha_i^+ + \alpha_i^-\right),
\end{align*}
which can be solved analytically. 

The crucial step is to select $I^\ell$, and this is done by finding a first index that does not satisfy the KKT 
conditions and a second one with some heuristic. It can be proven that, if at least one of the two violates the KKT conditions, then the objective 
is strictly decreased and convergence is obtained. Moreover, the vectors $\alpha^+=\alpha^-=0\in\R^n$ are always feasible and can thus be used as a 
first guess.
In practice, the iteration is stopped when a sufficiently small value of the cost function is reached.

In the case of SVR without offset discussed in the previous section the situation is even simpler, since the second constraint in \eqref{eq:smo_cond} is not present 
and it is thus possible to update a single pair $\left(\alpha_i^+, \alpha_i^-\right)$ at each iteration. Nevertheless, it has been proven in \cite{Steinwart2011} that 
using also in this case two indexes improves significantly the speed of convergence. Moreover, the same paper introduces several additional details to select the 
pair, to 
optimize the restricted cost function, and to establish termination conditions.

A general version of SMO for SVR is summarized in Algorithm \ref{alg:smo}, where we assume that the function $\eta:\{1,\dots, n\}\to \{1,\dots, n\}$ implements the 
selection rule of $I^\ell$.

\begin{algorithm}[h!t]
  \caption{SMO}
  \begin{algorithmic}[1]
      \State{Input: training set $X_{n}\in\Omega^{n}$, $Y_{n}\in\R^{n}$, kernel $K$ (and kernel parameters), regularization parameter $\lambda\geq 
0$, tube width $\varepsilon>0$, selection rule $\eta$, tolerance $\tau$.}
      \State{Set $\ell:=0$ and $\left(\alpha^+, \alpha^-\right)^{(0)} := (0,0)$.}    
        \While{$\left(\alpha^+, \alpha^-\right)^{(\ell)}$ does not satisfy KKT conditions within tolerance $\tau$.}
        \State{Set $\ell=\ell+1$.}
	\State{Set $I^\ell:=\{i, j\}:=\eta\left(\{1,\dots, n\}\right)$.}
	\State{Set $\left(\alpha_k^+, \alpha_k^-\right)^{(\ell)}:= \left(\alpha^+_k, \alpha_k^-\right)^{(\ell-1)}$ for $k\notin I^\ell$.}
	\State{Solve the optimization problem restricted to $I^\ell$.}
        \EndWhile
	\State{Set $I_N:=\{i : \alpha^-_i\neq 0 \text{ or } \alpha^+_i\neq 0\}$.}
        \State{Set $\alpha_i:=(\alpha_i^--\alpha^+_i)/2$ for $i\in I_N$.}
        \State{Output: $\alpha\in\R^N$, $I_N$.}
    \end{algorithmic}\label{alg:smo}
\end{algorithm}

\begin{remark}[Reference implementations]
We remark that there exists commonly used implementations of SVR (and other SVM-related algorithms), which are available in several programming languages and implement 
also some version of this algorithm. We mention 
especially LIBSVM\footnote{\url{https://www.csie.ntu.edu.tw/~cjlin/libsvm/}} \cite{libsvm} and liquidSVM\footnote{\url{http://www.isa.uni-stuttgart.de/software/}} 
\cite{liquidsvm}.
\end{remark}

\section{Model analysis using the surrogate}\label{sec:model_analysis}
Apart from predicting new inputs with good accuracy and a significant speedup, the surrogate model can be used to perform a variety of different tasks related to
meta-modeling, such as uncertainty quantification and state estimation. This can be done in a non-intrusive way, meaning that the full model is 
employed as a black-box that provides input-output pairs to train the surrogate, but is not required to be modified.

In principle, any kind of analysis that requires multiple evaluations can be significantly accelerated by the use of a surrogate, including the ones 
that are not computationally feasible due to the high computational cost of the full model. An example is uncertainty quantification, where the expected value of $f$ can 
be approximated by a Monte Carlo integration of $s$ using a set $X_m\subset\Omega$ of integration points, i.e., 
\begin{align*}
\int_{\Omega} f(x) dx \approx \frac{1}{m} \sum_{i=1}^m s(x_i).
\end{align*}
Once the surrogate is computed using a training set $(X_n, Y_n)$, this approximate integral can be evaluated also for $m\gg n$ with a possibly very small cost, since the 
evaluation of $s$ is significantly cheaper than the one of $f$.

Another example, which we describe in detail in the following, is the solution of an inverse problem to estimate the input parameter which generated a given output,  
i.e., from a given  
vector $\overline y\in\R^q$ we want to estimate $x\in\Omega$ such that $f(\overline x) = \overline y$.
This can be done by considering a state-estimation cost function $C:\Omega\to\R$ defined by
\begin{align}\label{eq:cost_state_estim}
C(x):= \frac{1}{2 \norm{2}{\overline y}^2} \norm{2}{s(x) - \overline y}^2,
\end{align}
and estimating $\overline x$ by the value $x^*$ defined as 
\begin{align*}
x^* := \min_{x\in\Omega} C(x).
\end{align*}
In principle, we could perform the same optimization also using $f$ instead of $s$ in \eqref{eq:cost_state_estim}, but the surrogate allows a rapid 
evaluation of $C$. Moreover, if $K$ is at least differentiable, then also $s$ is differentiable, and thus we can use gradient-based methods to minimize $C$. 

To detail this approach, we assume $f:\Omega\to\R^q$ and to have a surrogate obtained as in Section \ref{sec:greedy} with the separable matrix valued kernel of Section 
\ref{sec:mat_val}, i.e., from  \eqref{eq:greedy_surrogate} we have
\begin{align*}
s_N(x)= \sum_{k\in I_N}  \alpha_k K(x, x_k). 
\end{align*}
As explained in \eqref{eq:mat_val_coeffs}, in the vector-valued case $q>1$ we always assume that the output $s_N(x)$ and the coefficients $\alpha_k$ are row 
vectors, and in particular $\alpha \in \R^{N\times q}$ and $s_N(x)\in \R^{1\times q}$. 
In this case we have the following.

\begin{proposition}[Gradient of the state estimation cost]\label{prop:gradient_of_the_cost}
For $x\in\Omega\subset\R^d$ and $\overline y\in\R^q$, the gradient of the cost \eqref{eq:cost_state_estim}  can be computed in $x\in\Omega$ as
\begin{align*}
\nabla C(x) = \frac{1}{\norm{2}{\overline y}^2} (D \alpha) E^T,
\end{align*}
where $D\in\R^{d\times N}$ with columns $D_{j}:= \nabla_{x} K(x, x_j)$, and $E:= s_N(x)-\overline y\in \R^{1\times q}$.
\end{proposition}
\begin{proof}
By linearity, the gradient of $s_N$ in $x$ can be computed as  
\begin{align*}
\nabla s_N(x) = \sum_{j=1}^n \alpha_j \nabla_x K(x, x_j) = D \alpha \in \R^{d\times q}, 
\end{align*}
and thus 
\begin{align*}
\nabla C(x)  
&= \frac{1}{\norm{2}{\overline y}^2} \left(s_N(x)-\overline y \right)\nabla_x \left(s_N(x)-\overline y \right)
 = \frac{1}{\norm{2}{\overline y}^2} \left(s_N(x)-\overline y \right)\nabla s(x)\\
&= \frac{1}{\norm{2}{\overline y}^2} (D \alpha) E^T.
\end{align*}
\end{proof}
Observe in particular that whenever $K$ is known in closed form the matrix $D$ can be explicitly computed, and thus the gradient can be assembled using only 
matrix-vector multiplications of matrices of dimensions $N, d, q$, but independent of $n$. The solution $x^*$ can then be computed by any gradient-based 
optimization method, and each 
iteration can be performed in an efficient way.

\section{Parameter and model selection}\label{sec:parameter_selection}
For all the methods that we have seen the approximation quality of the surrogate depends on several  parameters, which need to be carefully chosen to obtain good 
results. 
There are both parameters defining the kernel, such as the shape parameter $\gamma>0$ in a RBF kernel, and model parameters such as the regularization 
parameter 
$\lambda\geq0$. To some extent, also the selection of the kernel itself can be considered as a parametric dependence of the model.
Moreover, it is essential to test the quality of the surrogate on an independent test set of data, since tuning it on the training set alone can very likely lead  
to overfitting, i.e., to obtain a model that is excessively accurate on the training set, while failing to generalize its prediction capabilities to unseen data.

In practical applications the target function $f$ is unknown, so it cannot be used to check if the approximation is good, and all we know is the training set 
$(X_n, Y_n)$. In this case the most common approach is to split the sets into train, validation and test sets in the following sense. We 
permute $(X_n, Y_n)$, fix numbers $n_{tr}$, $n_{val}$, $n_{te}$ such that $n = n_{tr}+n_{val}+n_{te}$, and define a partition of the dataset as 
\begin{align*}
X_{tr} &:= \{x_i, 1\leq i \leq n_{tr}\}\\
X_{val} &:= \{x_i, {n_{tr} + 1}\leq i\leq n_{tr} + n_{val}\}\\
X_{te} &:= \{x_i, {n_{tr} + n_{val} + 1}\leq i\leq n\},
\end{align*}
and similarly for $Y_{tr}$, $Y_{val}$, $Y_{te}$.

The idea is then to use the validation set $\left(X_{val},Y_{val}\right)$ to validate (i.e., choose) the parameters, and the test set $\left(X_{te}, Y_{te}\right)$ to 
evaluate the error. Having disjoint 
sets allows to have a fair way to test the algorithm. 

For the process we also need an error function that returns the error of the surrogate $s$ evaluated on a generic set of points $X:=\{x_i\}_i\subset\Omega$ 
w.r.t. the exact values $Y:=\{y_i\}_i$. We denote as $|X|$ the number of elements of $X$. Common examples are the maximal error and the Root Mean Square Error (RMSE) 
defined as
\begin{align}\label{eq:validation_error_fun}
E\left(s, X, Y\right) := \max_{1\leq i\leq |X|} \|s(x_i) - y_i\|_2\;\text{ or }\;E\left(s, X, Y\right) := \sqrt{\frac1{|X|} \sum_{i=1}^{|X|} \left\|s(x_i) - 
y_i\right\|_2^2}.
\end{align}

Then one chooses a set of possible parameter instantiations $\{p_1, \dots, p_{n_{p}}\}$, $n_p\in\N$ that has to be checked. A common choice for positive numerical 
parameters is 
to take them 
logarithmically equally spaced, since the correct scale is not known in advance, in general.

The training and validation process is described in Algorithm \ref{alg:validation}, where we denote as $s(p_i)$ the surrogate obtained with parameter 
$p_i$. It works as an outer loop with respect to the training of any of the surrogates that we have considered, and it has thus to be understood as part of the offline 
phase.

\begin{algorithm}[h!t]
  \caption{Model selection by validation}
  \begin{algorithmic}[1]
      \State{Input: $X_{{tr}}, X_{{val}}, X_{{te}}$,$Y_{{tr}}, Y_{{val}}, Y_{{te}}$, $\{p_1, \dots, p_{n_{p}}\}$}
        \For{$i = 1, \dots, n_{p}$}
	\State{Train surrogate $s(p_i)$ with data $(X_{tr}, Y_{tr})$}
	\State{Compute validation error $e_i := E(s(p_i), X_{val}, Y_{val})$}
        \EndFor
     \State{Choose parameter $\bar p := p_i$ with $i:= \argmin e_i$}    
     \State{Train surrogate $s(\bar p)$ with data $(X_{tr}\cup X_{val}, Y_{tr}\cup Y_{val})$}
     \State{Compute test error $\bar E = E(s(\bar p), X_{te}, Y_{te})$}
     \State{Output: surrogate $s(\bar p)$, optimal parameter $\bar p$, test error $\bar E$}
    \end{algorithmic}\label{alg:validation}
\end{algorithm}

A more advanced way to realize the same idea is via $k$-fold cross validation. To have an even better selection of the parameters, one can repeat the validation step 
(lines 2-6 in the previous algorithm) by changing the validation set at each step. To do so we do not select a validation set (so  $n = n_{tr}+n_{te}$), 
and instead consider a partition of $X_{tr}, Y_{tr}$ into a fixed number $k\in\{1,\dots, n_{tr}\}$ of disjoint subsets, all approximately of the same size, i.e., 
\begin{align*}
X_{tr} &:= \{x_i, 1\leq i \leq n_{tr}\} := \cup_{i=1}^k X_{i}\\
X_{te} &:= \{x_i, {n_{tr} + 1}\leq i\leq n\},
\end{align*}
and similarly for $Y_{tr}:= \cup_{i=1}^k Y_{i}$ and for $Y_{te}$. In the validation step each of the the $X_i$ is used as a validation set, and the validation is 
repeated for 
all $i=1,\dots, k$. 
In this case the error $e_i$ for the parameter $p_i$ is defined as the average error over all these permutations, as described in Algorithm \ref{alg:k_validation}.

\begin{algorithm}[h!t]
  \caption{Model selection by $k$-fold cross validation}
  \begin{algorithmic}[1]
      \State{Input: $X_{tr} = \cup_{i=1}^k X_{i}, X_{{te}}$, $Y_{tr} = \cup_{i=1}^k Y_{i}, Y_{{te}}$, $\{p_1, \dots, p_{n_{p}}\}$}
        \For{$i = 1, \dots, n_{p}$}
	 \For{$j=1, \dots, k$}
	\State{Train surrogate $s(p_i)$ with data $\left(\cup_{\ell\neq j} X_{\ell}, \cup_{\ell\neq j} Y_{\ell}\right)$}
	\State{Compute error $e^{(j)} := E(s(p_i), X_{j}, Y_{j})$}
	\EndFor
	\State{$e_i := \text{mean}\{e^{(j)}, 1\leq j\leq k\}$}
     \EndFor
     \State{Choose parameter $\bar p := p_i$ with $i:= \argmin e_i$}    
     \State{Train surrogate $s(\bar p)$ with data $(X_{tr}, Y_{tr})$}
     \State{Compute test error $\bar E = E(s(\bar p), X_{te}, Y_{te})$}
     \State{Output: surrogate $s(\bar p)$, optimal parameter $\bar p$, test error $\bar E$}
    \end{algorithmic}\label{alg:k_validation}
\end{algorithm}

We remark that, in the extreme case $k=N$, this $k$-fold cross validation is usually called Leave One Out Cross Validation (LOOCV).

\section{Numerical examples}\label{sec:numerics}
For the testing and illustration of the two methods of Section \ref{sec:interp} and Section \ref{sec:svr}, we consider a real-world application dataset describing the 
biomechanical modeling of the human spine introduced and studied in \cite{Wirtz2015a}. We refer to that paper for further details and we just give a 
brief description in the following.

The input-output function $f:\R^3\to\R^3$ represents the coupling between a global multibody system (MBS) and a Finite Elements (FEM) submodel. The human 
spine is represented as a MBS consisting of the vertebra, which are coupled by the interaction through intervertebral disks (IVDs). The PDE 
representing the behavior of each IVD is approximated by a FEM discretization, and it has the input geometry parameters as boundary conditions, and computes the output 
mechanical response as a result of the simulation. In particular, the three inputs are two spatial displacements and an angular inclination of a vertebra, and 
the three outputs 
are the corresponding two force components and the momentum which are transfered to the next vertebra. 
The dataset is generated by running the full model for $n := 1370$ different input parameters $X_n$ and generating the corresponding set of outputs $Y_n$. 

The dataset, as described in Section \ref{sec:parameter_selection}, is first randomly permuted and then divided in training and test datasets $\left(X_{n_{tr}}, 
Y_{n_{tr}}\right)$, $\left(X_{n_{te}}, Y_{n_{te}}\right)$ with $n_{tr}:=1238$ and  $n_{te} = 132$, corresponding to roughly $90\%$ and $10\%$ of the data. We remark 
that the full model predicts a value $(0,0,0)^T$ for the input $(0,0,0)^T$ and this sample pair is present in the dataset. We thus manually include it in the training 
set independently of the permutation. The training and test sets can be seen in Figure \ref{fig:sample_of_dataset}.

\begin{figure}[h!]
\begin{center}
\begin{tabular}{cc}
\includegraphics[width=0.45\textwidth]{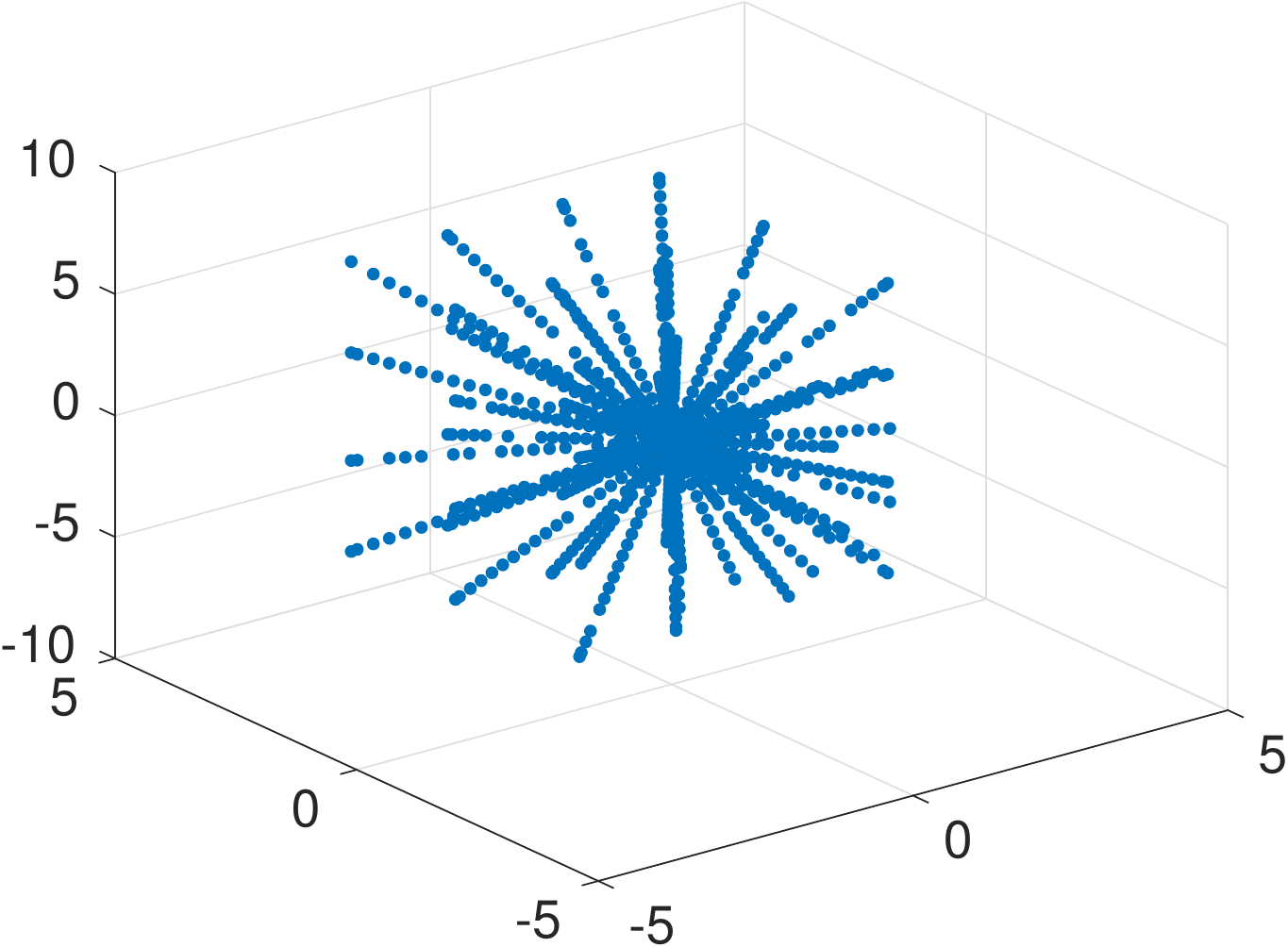} & \includegraphics[width=0.45\textwidth]{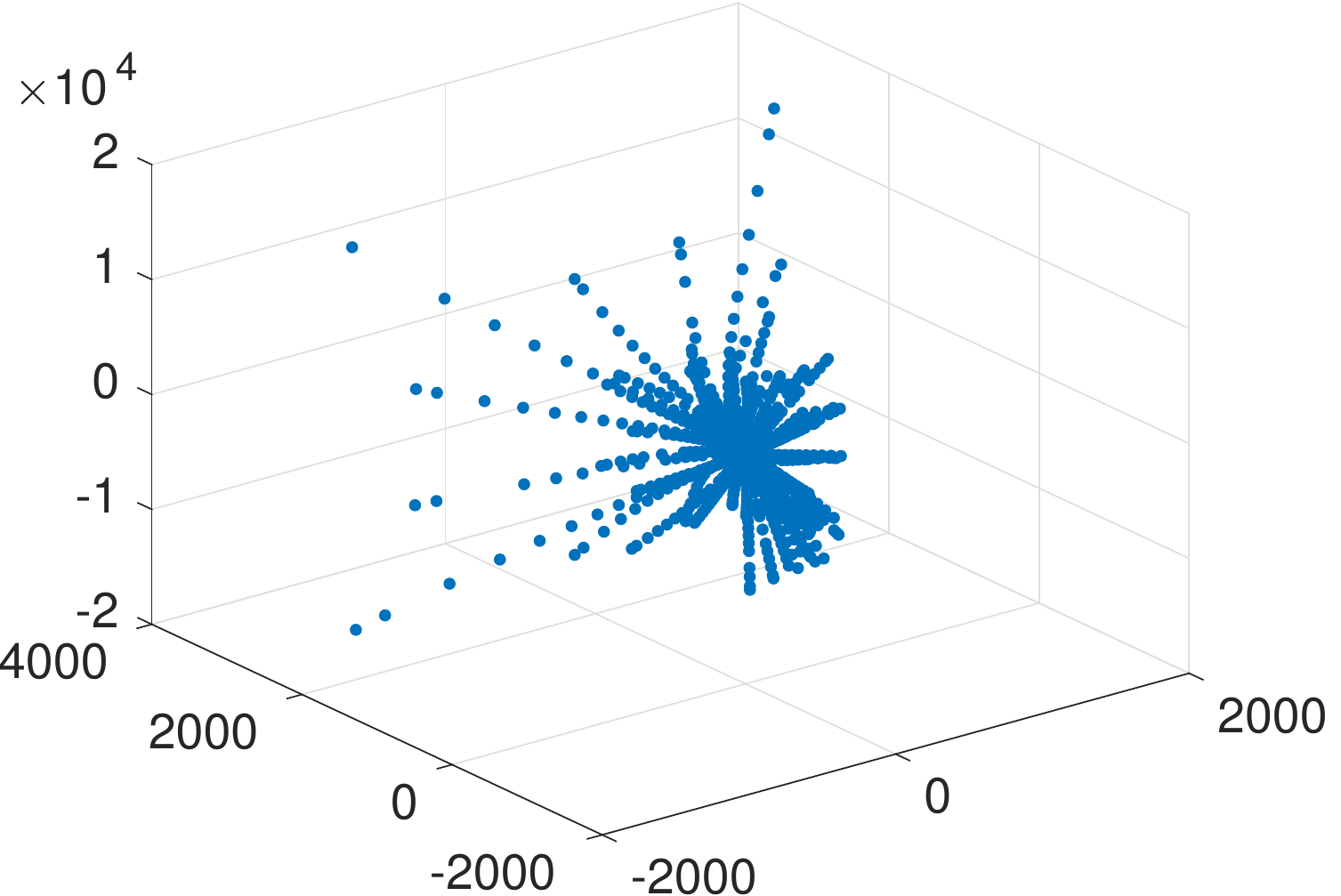}\\
\includegraphics[width=0.45\textwidth]{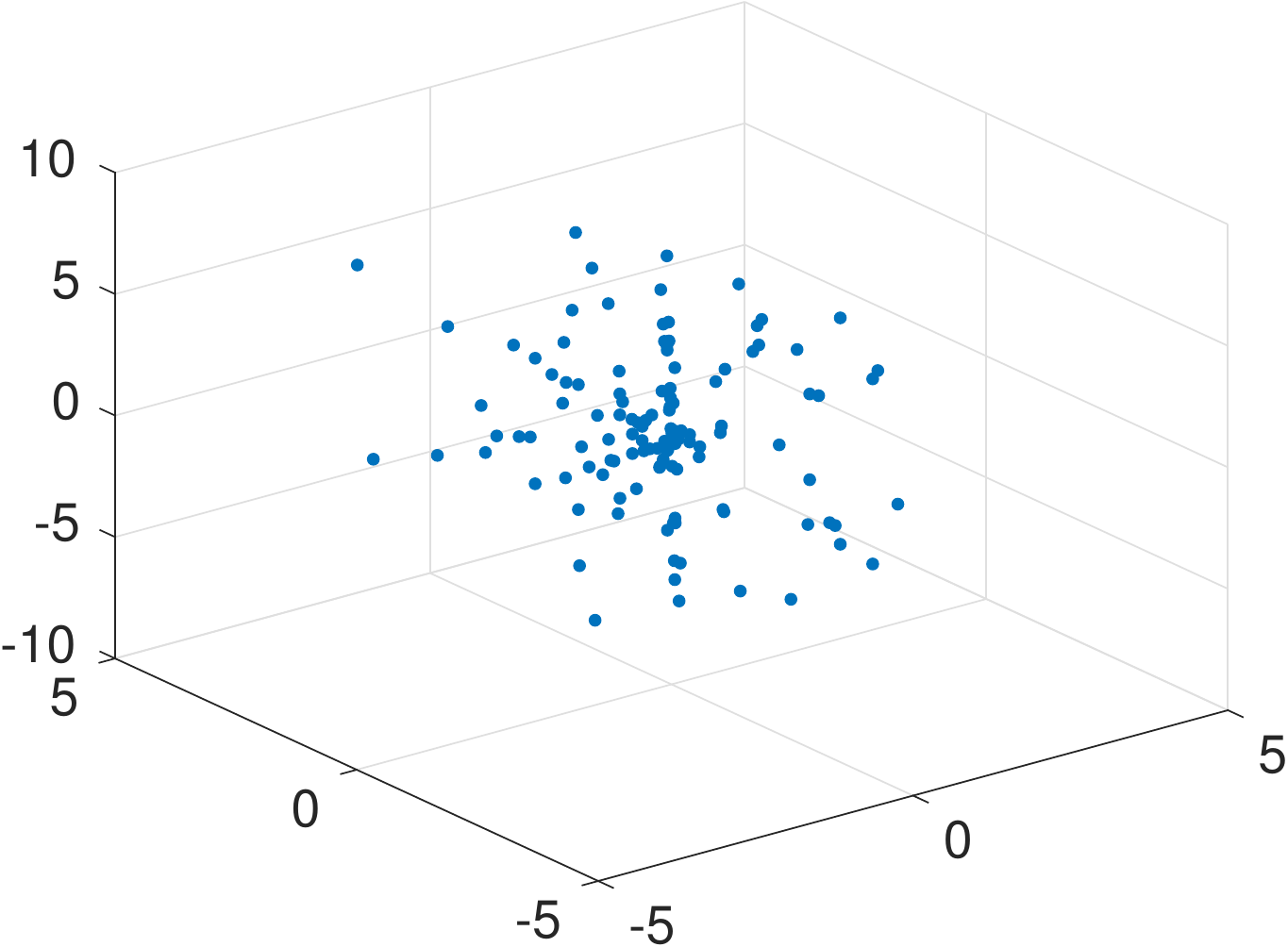} & \includegraphics[width=0.45\textwidth]{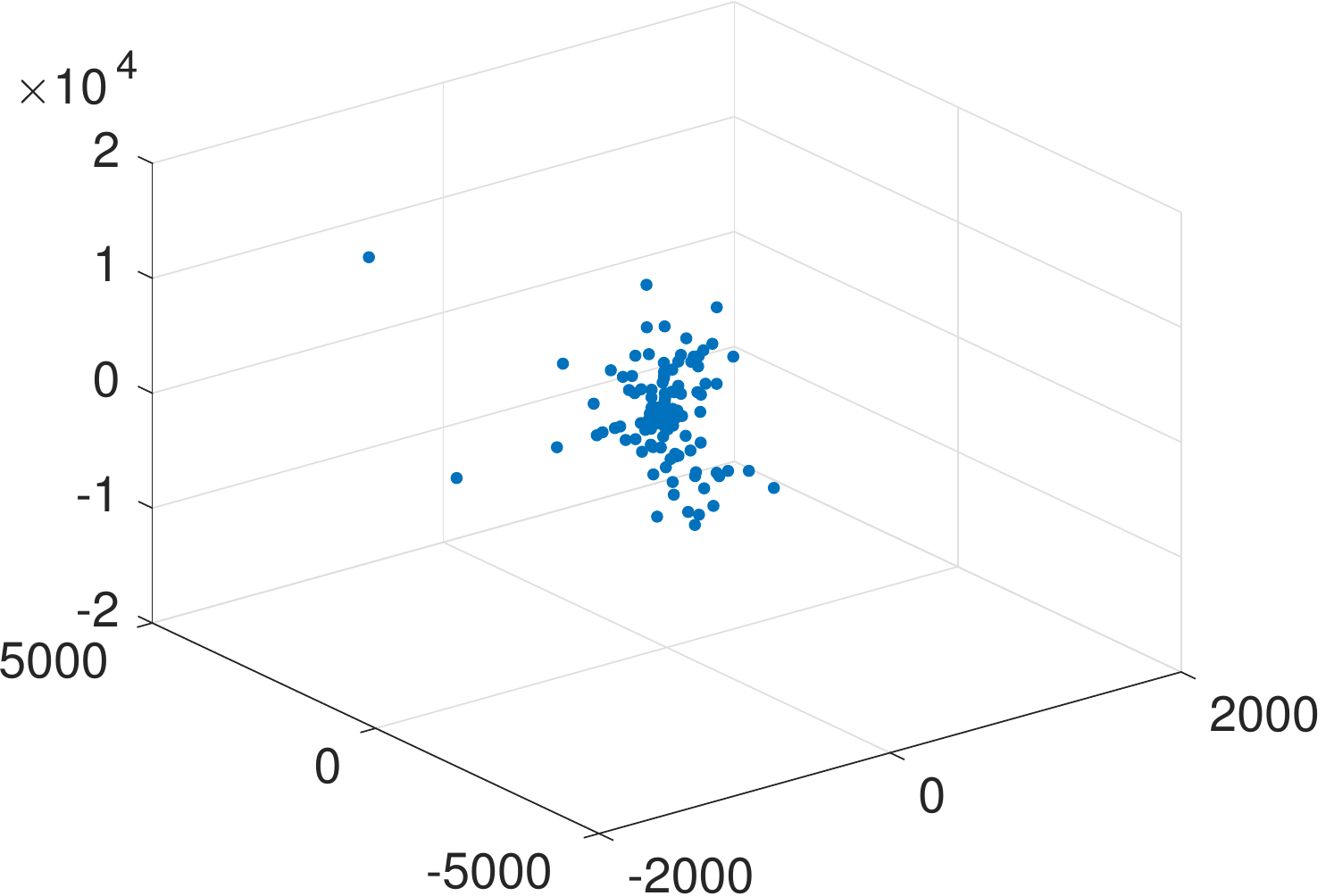}\\
\end{tabular}
\caption{Input parameters (left) and corresponding outputs (right) for the training (top row) and test set (bottom row).}\label{fig:sample_of_dataset}
\end{center}
\end{figure}

The models are trained using a Matlab implementation of the algorithms. For VKOGA we use an own 
implementation\footnote{\url{https://gitlab.mathematik.uni-stuttgart.de/pub/ians-anm/vkoga}}, while for SVR we employ the 
KerMor package\footnote{\url{https://www.morepas.org/software/kermor/index.html}}, which provides an implementation of the $2$-index SMO for the SVR without 
offset that is discussed in Section \ref{sec:smo}. We remark that this implementation requires the output data to be scaled in $[-1, 1]$, and thus we perform this 
scaling for the training and validation, while the testing is executed by scaling back the predictions to the original range. To 
have a fair comparison, we use the same data normalization also for the VKOGA models.

The regularized VKOGA (with $f$-, $P$-, and $f/P$-greedy selection rules) and the SVR models are trained with the Gaussian kernel. Both algorithms 
depend on the shape parameter $\gamma$ of the kernel and on the regularization parameter $\lambda$, while SVR additionally depends on the width $\varepsilon$ of the 
tube. These parameters are selected by $k$-fold cross validation as described in Section \ref{sec:parameter_selection}. The values of $k$ and of the parameter samples 
used for validation are reported in Table \ref{tab:val_par}, where each parameter set is obtained by generating logarithmically equally spaced samples in the 
given interval, i.e., $400$ parameter pairs are tested for VKOGA and $4000$ triples for SVR. As an error measure we use the $\max$ error in 
\eqref{eq:validation_error_fun}.
We remark that the SVR surrogate is obtained by training a separate model for each output, as described in Section 
\ref{sec:svr}, but only one cross validation is used. This means that for each parameter triple three models are trained, and then the parameter is evaluated in the 
prediction of the three dimensional output.

Moreover, the training of the VKOGA surrogates is terminated when the square of the power function is below the tolerance $\tau_P:=10^{-12}$, or when the training error 
is below the tolerance $\tau_f:=10^{-6}$. Additionally, it would be possible to use a maximal number of selected points as stopping criterion, and this offers the 
significant advantage of directly controlling the expansion size, which could be reduced to any given number (of course at the price of a reduced accuracy). In 
the case of SVR, instead, the number $N$ is a result of the tuning of the remaining parameters.

\begin{table}[h!]
\begin{center}
\begin{tabular}{|c|ccc|ccc|ccc|}
\hline
$k$& $\gamma_{\min}$&$\gamma_{\max}$&$n_{\gamma}$& $\lambda_{\min}$&$\lambda_{\max}$&$n_{\lambda}$&$\varepsilon_{\min}$&$\varepsilon_{\max}$&$n_{\varepsilon}$\\
\hline
$5$& $10^{-2}$& $10^{1}$& $20$&$10^{-16}$& $10^{3}$& $20$&$10^{-10}$& $10^{-3}$& $10$ \\
\hline
\end{tabular}
\end{center}
\caption{Parameters ranges and sample numbers used in the $k$-fold cross validation.}\label{tab:val_par}
\end{table}

In Table \ref{tab:train_results} we report the values of the parameters selected by the validation procedure for the four models, as well as the number $N$ of nonzero 
coefficients in the trained kernel expansions. Observe that for SVR the three values of $N$ refer to the number of support vectors for the three scalar-valued models. 
Moreover, the number of support vectors or kernel centers is only slightly larger for SVR than for the VKOGA models, but, as discussed in the following, the VKOGA 
models give prediction errors which are up to two orders of magnitude smaller than the ones of the SVR model. 
\begin{table}[h!]
\begin{center}
\begin{tabular}{|l|c|ccc|}
\hline
Method                     & $N$     & $\bar\gamma$         & $\bar\lambda$  & $\bar\varepsilon$\\
\hline
VKOGA $P$-greedy           & $1000$  & $4.9 \cdot 10^{-2}$  & $10^{-11}$     & --\\
VKOGA $f$-greedy           & $879$   & $4.3 \cdot 10^{-2}$  & $10^{-11}$     & --\\
VKOGA $f/P$-greedy         & $967$   & $6.2 \cdot 10^{-2}$  & $10^{-9}$      & --\\
SVR, output 1              & $359$  & $1.8 \cdot 10^{-1}$  & $10^{2}$       & $7.7\cdot10^{-7}$\\
\hspace{.7cm} output 2     & $378$  &                      &                &    \\
\hspace{.7cm} output 3     & $405$  &                      &                &    \\
\hline
\end{tabular}
\end{center}
\caption{Selected parameters and number of nonzero coefficients in the kernel expansions.}\label{tab:train_results}
\end{table}

We can now test the four models in the prediction on the test set. Table \ref{tab:test_results} contains various error measures between the prediction of the surrogates 
and the exact data. We report the values of the maximum error $E_{\max}$ and the RMSE $E_{RMSE}$ defined in \eqref{eq:validation_error_fun}, and the relative maximum 
error $E_{\max, rel}$ obtained by scaling each error by the norm of the exact output.  

\begin{table}[h!]
\begin{center}
\begin{tabular}{|l|ccc|}
\hline
Method                     & $E_{\max}$       & $E_{RMSE}$            & $E_{\max, rel}$ \\
\hline
VKOGA $P$-greedy           & $1.6\cdot 10^2$   & $22.3$               & $2.2 \cdot 10^{-1}$\\
VKOGA $f$-greedy           & $1.6\cdot 10^2$   & $22.4$               & $2.0 \cdot 10^{-1}$\\
VKOGA $f/P$-greedy         & $1.6\cdot 10^2$   & $23.2$               & $8.8 \cdot 10^{-1}$\\
SVR                        & $1.3\cdot 10^3$   & $1.6 \cdot 10^{2}$   & $1.4 \cdot 10^{1}$      \\
\hline
\end{tabular}
\end{center}
\caption{Test errors: maximum error $E_{\max}$, RMSE error $E_{RMSE}$, maximum relative error $E_{\max, rel}$}\label{tab:test_results}
\end{table}

To provide a better insight in the approximation quality of the methods, we show in Figure \ref{fig:errors} the distribution of the error over the test set. The plots 
show, for each sample $(x_i, y_i)$ in the test set, the absolute error $\norm{2}{y_i-s(x_i)}$ as a function of the magnitude $\norm{2}{y_i}$ of the output. Moreover, the 
black lines represent a relative error from $10^0$ to $10^{-3}$. It is clear that in all cases the maximum and RMS errors of Table \ref{tab:test_results} are dominated 
by the values obtained for outputs of large norm, where the VKOGA models obtain a much better accuracy than SVR. 
The relative errors, on the other hand, are not evenly distributed for SVR, where most of the test set is approximated with a relative error between $10^1$ 
and $10^{-2}$ except for the samples with small magnitude of the output. For these data, the model gives increasingly bad predictions as the magnitude is smaller, 
reaching a relative error much larger than $1$. 
The VKOGA models, instead, obtain a 
relative error smaller than $10^{-2}$ on the full test set except for the entries of small magnitude. For these 
samples, the $f$- and $P$-greedy versions of the algorithm perform almost the same and better than the $f/P$-greedy variant, thus giving an overall smaller 
relative error in Table \ref{tab:test_results}. Moreover, these results are obtained with a significantly smaller expansion size for $f$-greedy than for $P$-greedy. 
Indeed, even if the SVR surrogates for the individual output components are smaller than the VKOGA ones, the overall number of non-zero coefficients is $359+378+405 = 
1142$, i.e., more than the one of each of the three VKOGA models, thus leading to a less accurate and more expensive 
surrogate.  

\begin{figure}[h!]
\begin{center}
\begin{tabular}{cc}
\includegraphics[width=0.45\textwidth]{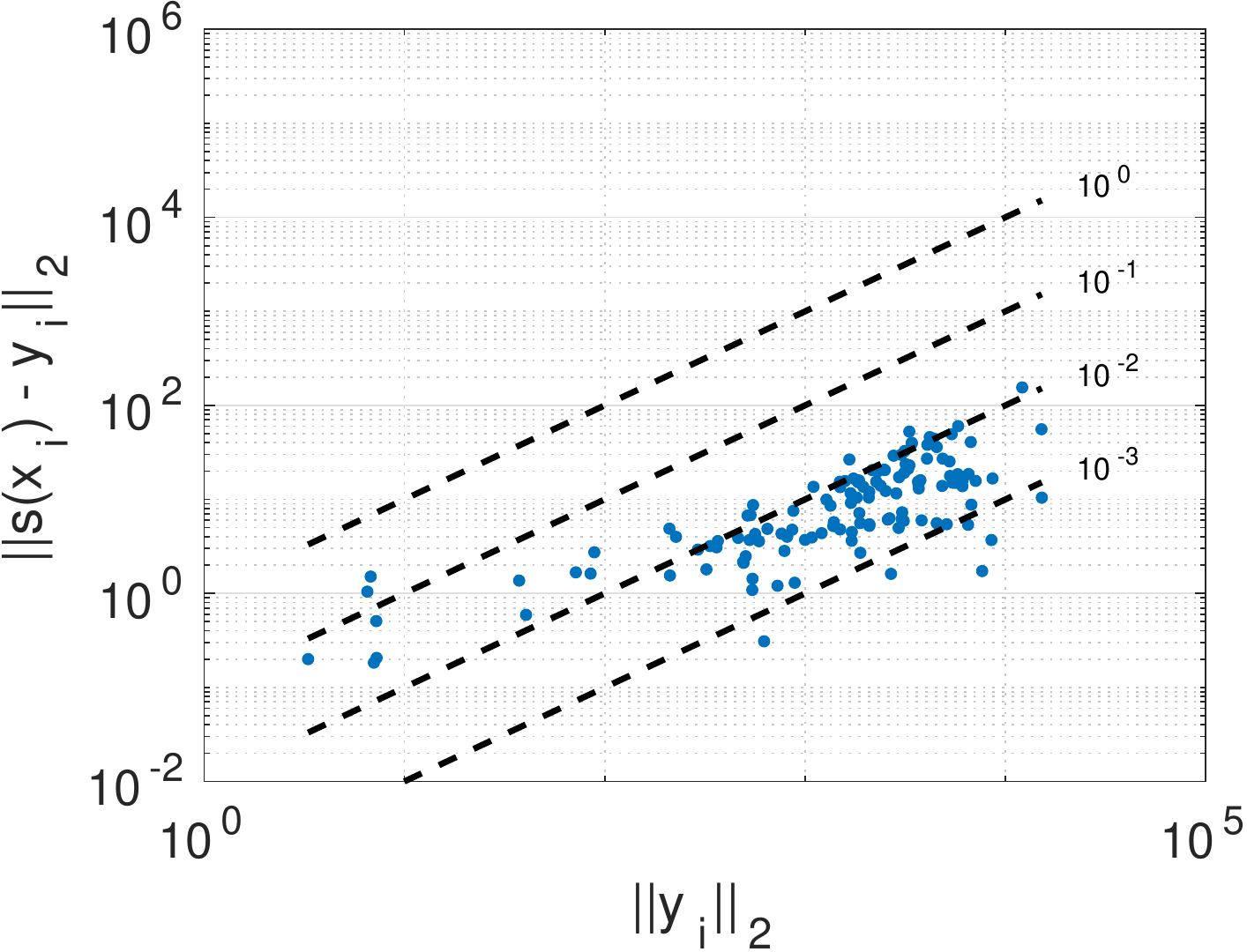} & 
\includegraphics[width=0.45\textwidth]{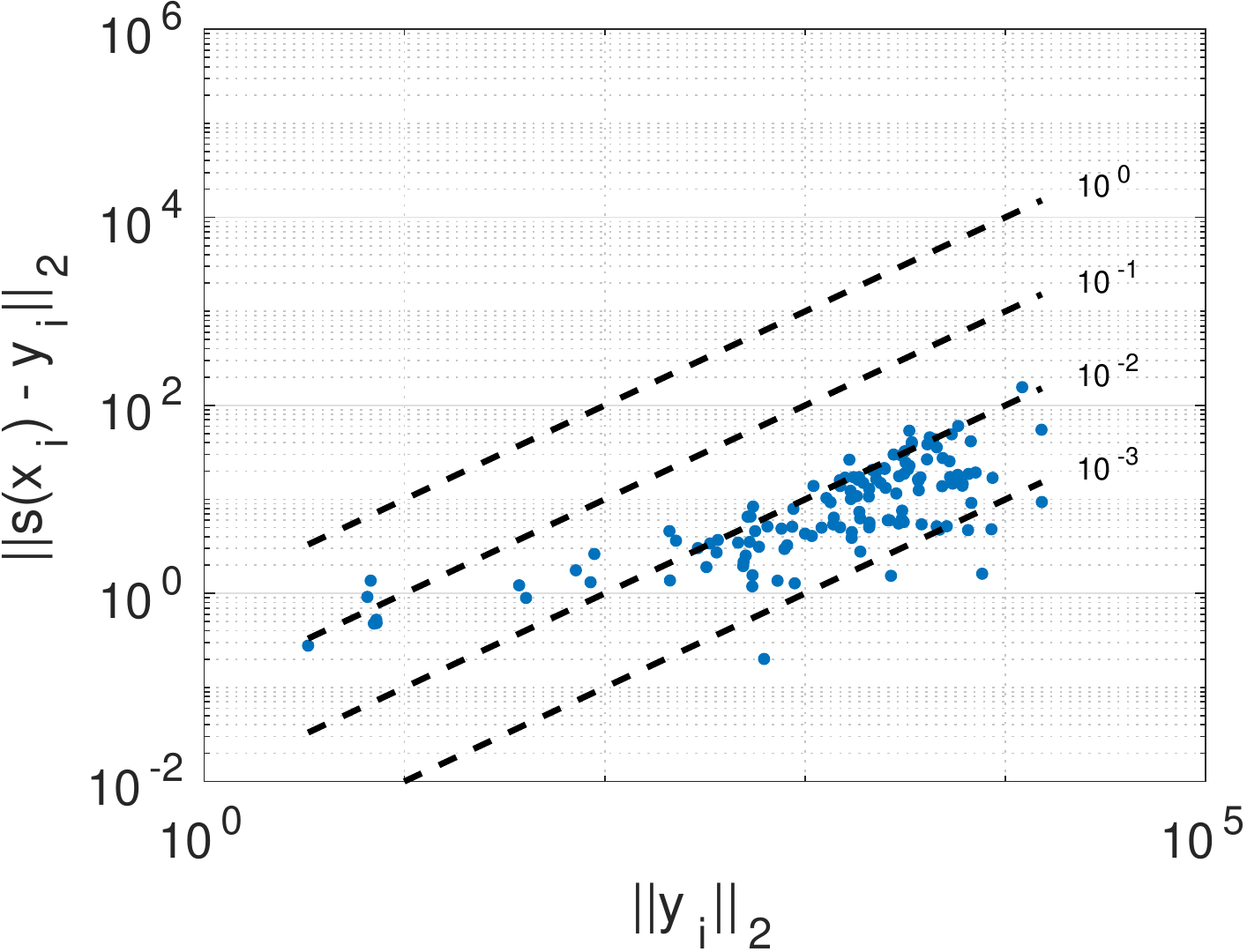}\\
\includegraphics[width=0.45\textwidth]{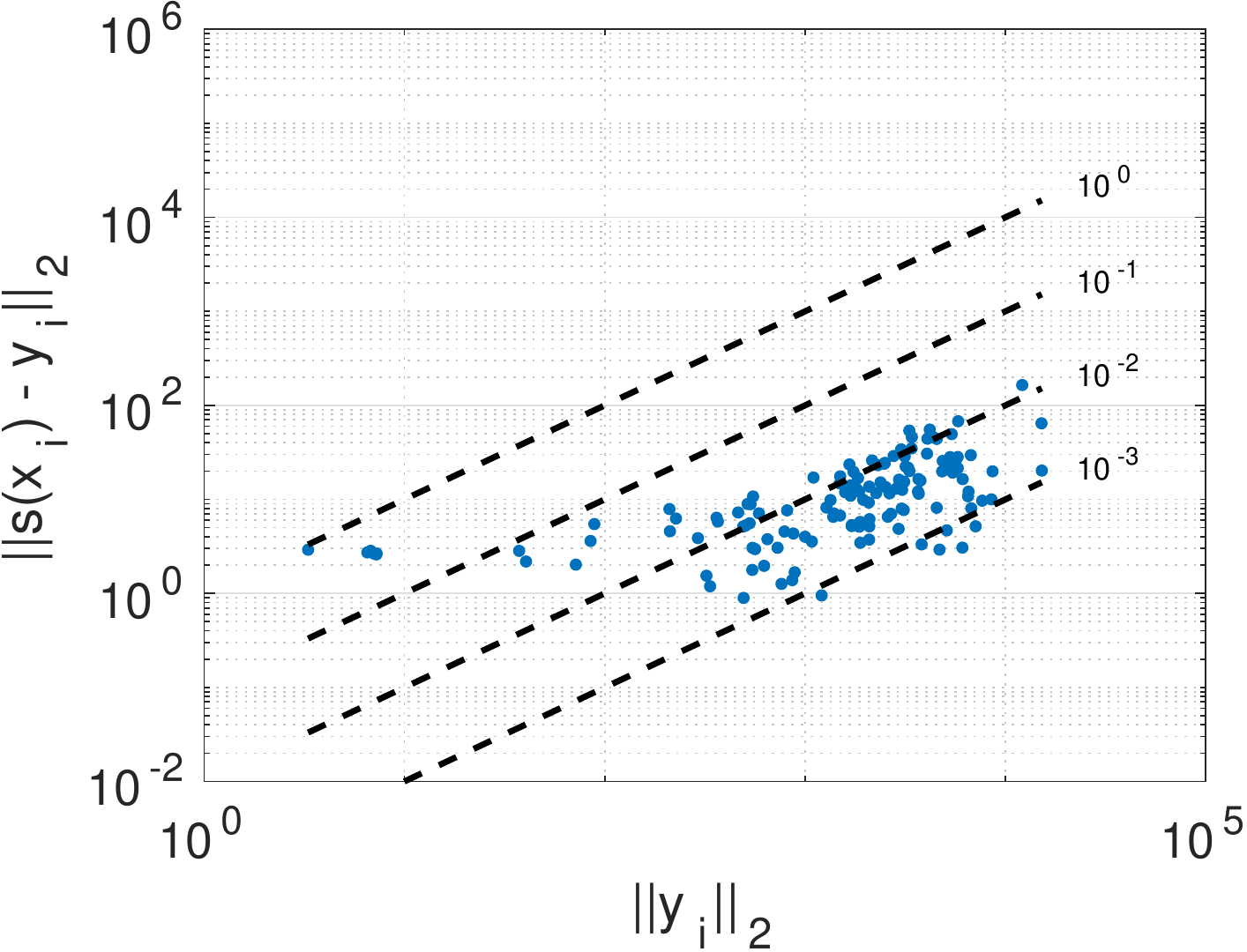} & 
\includegraphics[width=0.45\textwidth]{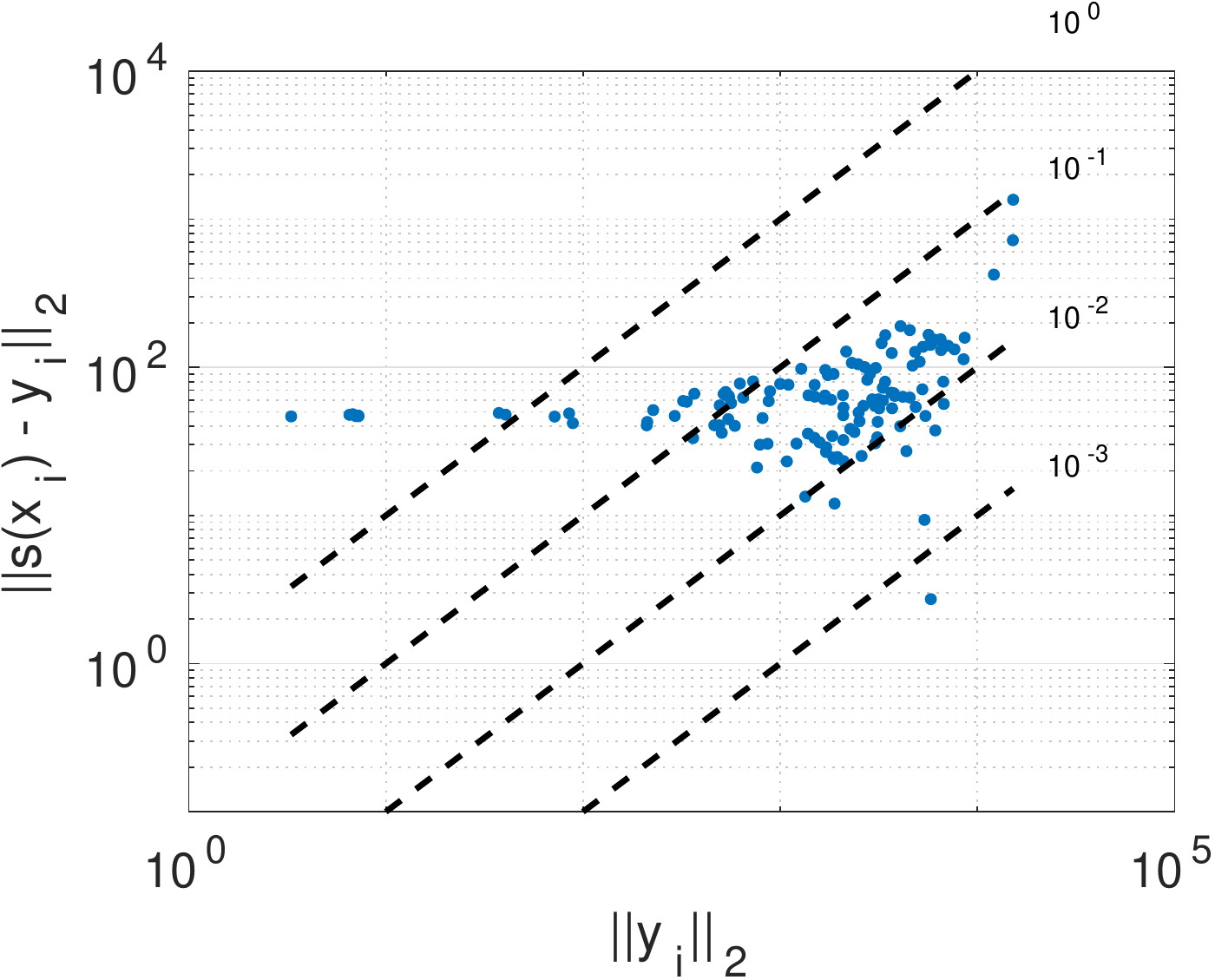}\\
\end{tabular}
\caption{Absolute errors as functions of the magnitude of the output, and relative error levels from $10^0$ to $10^{-3}$ for the 
surrogates obtained with $P$-greedy VKOGA (top left), $f$-greedy VKOGA (top right), $f/P$-greedy VKOGA (bottom left) and SVR (bottom right).}\label{fig:errors}
\end{center}
\end{figure}

Regarding the runtime requirements, we can now estimate both the offline (training) and online (prediction) times.
The offline time required for the validation and training of the models is essentially determined by the number of parameters tested in the $k$-fold cross validation, 
while the training time of a single model is almost negligible. As a comparison, we report in Table \ref{tab:offline_runtime} the average runtime $\tilde T_{offline}$ 
for $10$ runs of the 
training of the models for the fixed set of parameters of Table \ref{tab:train_results}. All the reported times are in the ranges of seconds (for VKOGA) and below one 
minute (for SVR). We remark that this timing is only a very rough indication and not a precise comparison, since the times highly depends on the number of selected 
points (for VKOGA) and the number of support vectors for SVR, and both are dependent on the used parameters. For example, we repeated the experiment for SVR 
with the same parameter set but with $\varepsilon=10^{-1}$. In this case this value of 
$\varepsilon$ is overly large (if compared to the the one selected by cross validation) and it likely produces a useless model, but nevertheless we obtain an 
average training time of $0.03$ sec. 

A more interesting comparison is the online time, which directly determines the efficiency of the surrogate models in the replacement of the full simulation. In this 
case, we evaluate the models $5000$ times on the full test set consisting of $n_{te} = 132$ samples, and we report the average online time $\tilde T_{online}$ per single 
test sample in Table \ref{tab:offline_runtime}. The table contains also again the number $N$ of elements of the corresponding kernel expansions, and it is evident that a 
smaller value leads to a faster evaluation of the model.

In the original paper \cite{Wirtz2015a}, it has been estimated that a $30$ sec full simulation with $24$ IVDs with a timestep $\Delta t = 10^{-3}$ sec requires 
$7.2 \cdot 10^5$ evaluations of the coupling function $f$, and these were estimated to require $600$ h. This corresponds to an average of 
$\tilde T_{full}=3$ sec per evaluation of $f$, giving a speedup $\tilde T_{full}/ \tilde T_{online}$ as reported in Table \ref{tab:offline_runtime}.

\begin{table}[h!]
\begin{center}
\begin{tabular}{|c|c|ccc|}
\hline
Method            & $N$      & $\tilde T_{offline}$  & $\tilde T_{online}$     & $\tilde T_{full}/ \tilde T_{online}$ \\
\hline
VKOGA $P$-greedy  & $1000$   &$1.67$ sec             & $9.97\cdot 10^{-6}$ sec & $3.01 \cdot 10^{5}$\\
VKOGA $f$-greedy  & $879$    &$1.41$ sec             & $9.44\cdot 10^{-6}$ sec & $3.18 \cdot 10^{5}$\\
VKOGA $f/P$-greedy& $967$    &$1.66$ sec             & $9.92\cdot 10^{-6}$ sec & $3.02 \cdot 10^{5}$\\
SVR   ($3$ models)& $1142$   &$52.0$ sec             & $2.28\cdot 10^{-5}$ sec & $1.32 \cdot 10^{5}$\\
\hline
\end{tabular}
\end{center}                    
\caption{Average offline time (training only), online time, and projected speedup factor for the four different models.}\label{tab:offline_runtime}
\end{table}

These surrogates can now be employed to solve different tasks that require multiple evaluations of $f$. As an example, we employ the $f$-greedy model (as the most 
accurate and most efficient) to solve a parameter estimation problem as described in Section \ref{sec:model_analysis}. 
We consider the output values $Y_{n_{te}}$ in the test set as a set of measures that have not been used in the training of the model, and we try to 
estimate the values of $X_{n_{te}}$. For each output vector $y_i\in\R^3$ we define a target value $\overline y:= y_i + \eta \norm{2}{y_i} v$ to define the cost 
\eqref{eq:cost_state_estim}, where $v\in\R^3$ is a uniform random vector representing some noise, and $\eta\in[0,1]$ is a noise level. We then use a built-in Matlab 
optimizer with the gradient of Proposition \ref{prop:gradient_of_the_cost}, with initial guess $x_0:=0\in\R^3$, to obtain an estimate $x_i^*$ of $x_i$. The results of 
the estimate for each output value in the test set are depicted in Figure \ref{fig:par_est} for $\eta = 0, 0.1$, where we report also the final value of the cost 
function $C(x_i^*)$. In all cases, the optimizer seems to converge, since the value of the cost function is in all cases smaller than $10^{-4}$, which represents a 
relative value smaller than $10^{-3}$ with respect to the magnitude of the input values.
The maximum absolute error in the estimations is quite uniform for all the samples in the test set, and this results in a good relative error of 
about $10^{-1}$ for large inputs, while for inputs of very small magnitude the relative error is larger than $1$, and a larger noise level leads to less accurate 
predictions. This behavior is coherent with the analysis of the test error discussed above, since the approximant is less accurate on inputs of small magnitude, and thus 
it provides a less reliable surrogate in the cost function.

\begin{figure}[h!]
\begin{center}
\begin{tabular}{ll}
\includegraphics[width=0.45\textwidth]{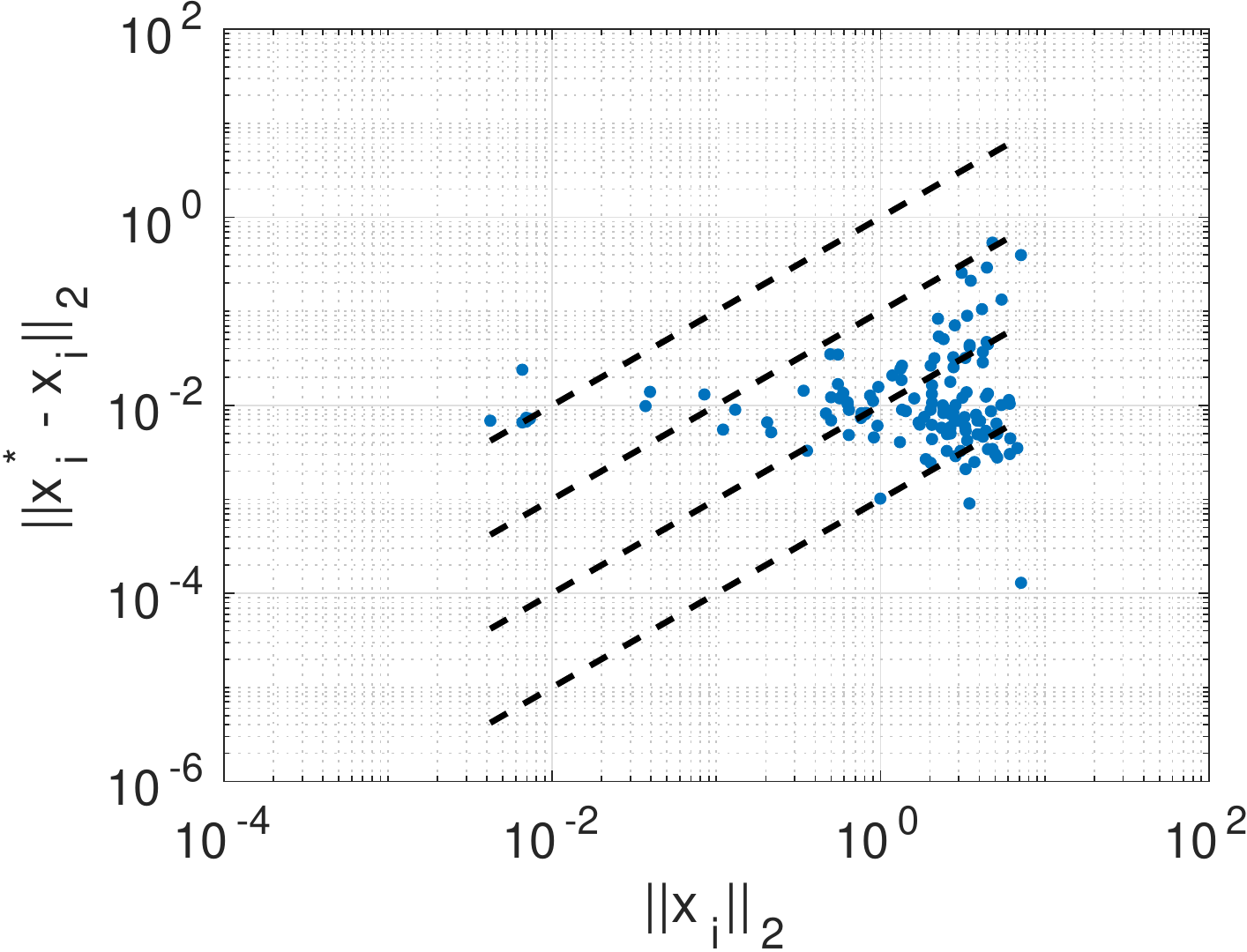} & 
\includegraphics[width=0.45\textwidth]{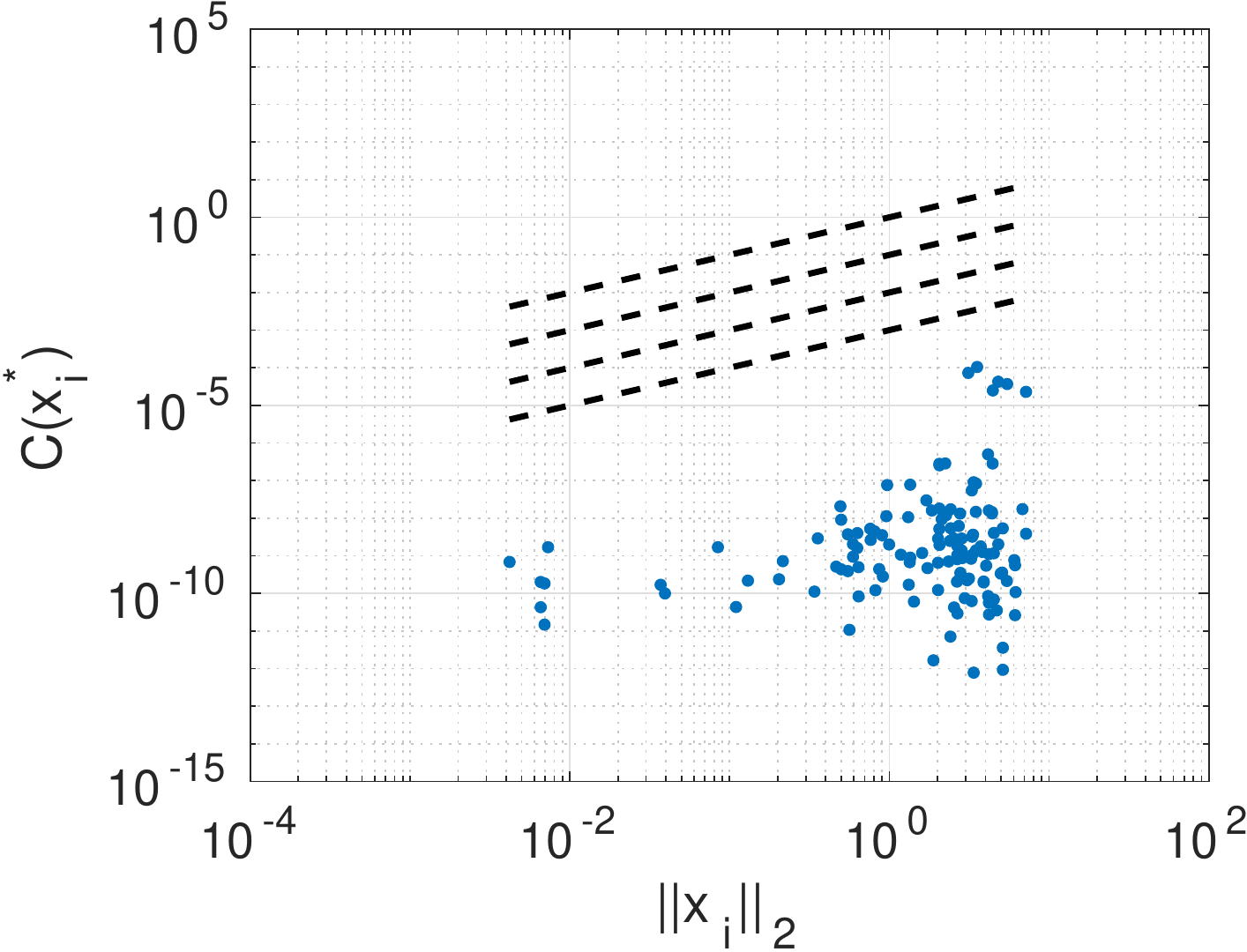}\\
\includegraphics[width=0.45\textwidth]{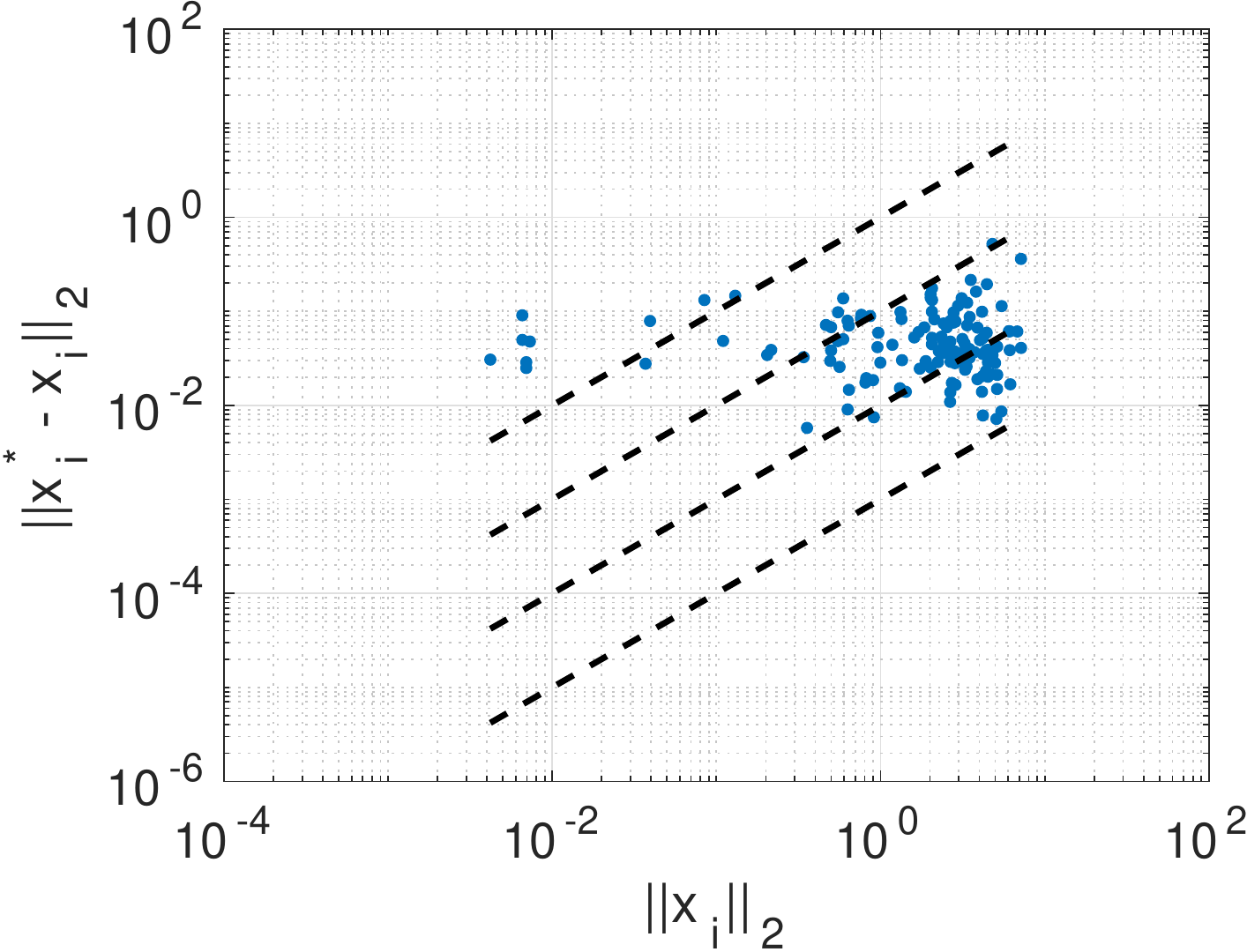}& 
\includegraphics[width=0.45\textwidth]{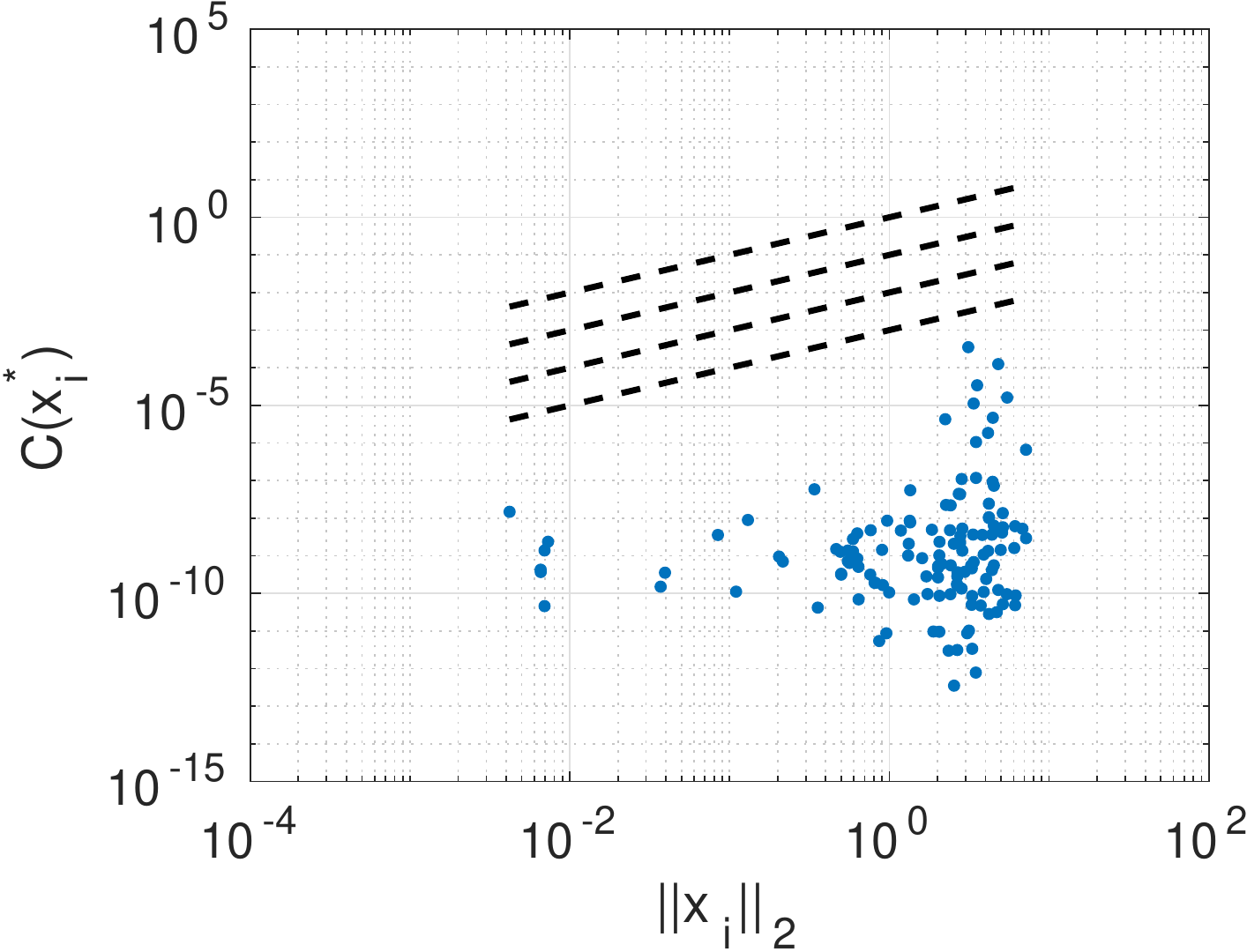}\\
\end{tabular}
\caption{Absolute errors of the input estimation as functions of the magnitude of the output (left), and value of the cost function at the estimated input (right) for a 
noise level $\eta = 0$ (top row) and $\eta = 0.1$ (bottom row) using the $f$-greedy VKOGA model. The dotted lines represent relative error levels from $10^0$ to 
$10^{-3}$.}\label{fig:par_est}
\end{center}
\end{figure}

\section{Summary and perspectives}
In this chapter we discussed the use of kernel methods to construct surrogate models based on scattered data samples. These methods can be applied to data with 
general structure, and they scale well with the dimension of the input and output values. In particular, we analyzed issues and methods to obtain sparse solutions, 
which are then extremely fast to evaluate, while still being very accurate. These properties have been further demonstrated on numerical tests on a real application 
dataset.
These methods can be analyzed in the common framework of Reproducing Kernel Hilbert Spaces, which provides solid theoretical foundations and a high 
flexibility to derive new algorithms.

The integration of machine learning and model reduction is promising and many interesting aspects have still to be investigated. For example, surrogate models 
have been used in \cite{Guo2018,Guo2019} to learn a representation with respect to projection-based methods, and generally a more extensive application of machine 
learning to 
dynamical systems requires additional understanding and the derivation of new techniques. Moreover, the field of data-based numerics is very promising, where 
classical 
numerical methods are integrated or accelerated with data-based models.

\bibliographystyle{abbrv}

\end{document}